\documentclass[10pt,a4paper,oneside,openany]{book}

\usepackage{latexsym}
\usepackage{amsfonts}
\usepackage{amssymb}
\usepackage{amsmath}
\usepackage{amsthm}
\usepackage{setspace} 
\usepackage{color}
\usepackage{hyperref}
\usepackage{pdfpages}
\usepackage{paralist}  
\usepackage{etoolbox} 
\patchcmd{\thebibliography}{\chapter*}{\section*}{}{}
\newcommand{\R}{\mathbb{R}}		% blackboard bold R
\newcommand{\beq}{\begin{equation}}		%starts equation with number
\newcommand{\eeq}{\end{equation}}			%ends equation with number
\newcommand{\beqq}{\begin{equation*}}	%starts equation without number
\newcommand{\eeqq}{\end{equation*}}		%ends equation without number

\newcommand{\id}{1\hspace{-0,9ex}1}

\renewcommand{\P}{\mathbb{P}}

\newcommand{\B}{\mathfrak{B}} 
\newcommand{\F}{\mathcal{F}} 
\newcommand{\A}{\mathcal{A}} 

\newcommand{\Var}{\text{Var}}

\newcommand{\x}{\text{\scalebox{0.62}{$\mathbb{X}$}}}
\newcommand{\X}{\mathbb{X}}

\allowdisplaybreaks

\newtheorem{theorem}{Theorem}[section]
\newtheorem{lemma}[theorem]{Lemma}
\newtheorem{definition}[theorem]{Definition}
\newtheorem{remark}[theorem]{Remark}

\newtheorem{corollary}[theorem]{Corollary}
\newtheorem{assumption}[theorem]{Assumption}
\newtheorem{example}[theorem]{Example}
\newtheorem{proposition}[theorem]{Proposition}

\newtheorem{stepp}{\noindent\bf{Step}}

{\begin{stepp}\rm}{\rm \end{stepp}}

\begin{document}
\pagestyle{headings} \thispagestyle{headings} \thispagestyle{empty}

\noindent\rule{15.812cm}{0.4pt}
\begin{center}
\textsc{A Markov Process Approach to the asymptotic Theory of abstract Cauchy Problems driven by Poisson Processes}
\end{center}
\begin{center}
	by
\end{center} 
\begin{center}
		\textsc{Alexander Nerlich\footnote{Affiliation: Ulm University}\footnote{Affiliation's address: 89081 Ulm, Helmholtzstr. 18, Germany}\footnote{Author's E-Mail: alexander.nerlich@uni-ulm.de}\footnote{Author's ORCID: 0000-0001-7823-0648}}
\end{center}
\noindent\rule{15.812cm}{0.4pt}
\vspace{0.4cm}
\pagestyle{myheadings} 
\begin{center} 
	{\large ABSTRACT} 
\end{center} 
In this paper, we employ Markov process theory to prove asymptotic results for a class of stochastic processes which arise as solutions of a stochastic evolution inclusion and are given by the representation formula
\begin{align*}
\mathbb{X}_{x}(t)=\sum \limits_{m=0}\limits^{\infty}T((t-\alpha_{m})_{+})(\text{\scalebox{0.62}{$\mathbb{X}$}}_{x,m})1\hspace{-0,9ex}1_{[\alpha_{m},\alpha_{m+1})}(t),
\end{align*}
where $(T(t))_{t \geq 0}$ is a (nonlinear) time-continuous, contractive semigroup acting on a separable Banach space $(V,||\cdot||_{V})$, $(\alpha_{m})_{m \in \mathbb{N}}$ is the sequence of arrival times of a homogeneous Poisson process, $x$ is a $V$-valued random variable and $(\text{\scalebox{0.62}{$\mathbb{X}$}}_{x,m})_{m \in \mathbb{N}}$ is a recursively defined sequence of $V$-valued random variables, fulfilling $\text{\scalebox{0.62}{$\mathbb{X}$}}_{x,0}=x$.\\
It will be demonstrated that $\mathbb{X}_{x}$ is, under some distributional assumptions on the involved random variables, a time-continuous Markov process and that it obeys, under polynomial decay conditions on $T$, the strong law of large numbers (SLLN) and, if the decay rate is sufficiently fast, also the central limit theorem (CLT). Finally, we consider two examples: A nonlinear ordinary differential equation and the (weighted) $p$-Laplacian evolution equation for $p \in (2,\infty)$.\\ 
\textbf{Mathematical Subject Classification (2010).} 60J25, 47J35, 60H15, 35B40, 60F05, 60F15 \\ 
\textbf{Keywords.} Markov Processes, Nonlinear evolution equation, Stochastic evolution inclusion, Pure jump noise, Asymptotic results, Strong law of large numbers, Central limit theorem, Weighted p-Laplacian evolution equation

\section{Introduction}

The theory of Markov processes is a beautiful tool to gain asymptotic results for stochastic processes. Particularly in the area of SPDEs, this theory enables one to get profound insights into the long time behavior of the SPDE's solution. In this paper, we apply Markov process theory to a class of stochastic processes which arise as solutions of abstract Cauchy problems driven by Poisson processes.\\ 

Let us embark on the endeavor ahead of us by rigorously describing the processes considered here: To this end, let $(V,||\cdot||_{V})$ be a separable Banach space and let $(T(t))_{t\geq 0}$ be a time-continuous, contractive semigroup on $V$. Moreover, let $(\Omega,\F,\P)$ be a complete probability space, let $(\eta_{m})_{m \in \mathbb{N}}$ and $(\beta_{m})_{m \in \mathbb{N}}$ be i.i.d. sequences that are independent of each other; where the former consists of $V$-valued random variables  and the latter of $(0,\infty)$-valued, exponentially distributed random variables. Moreover, we coin the term "independent initial" as a $V$-valued random variable $x$ which is independent of $((\beta_{m})_{m \in \mathbb{N}},(\eta_{m})_{m \in \mathbb{N}})$ and introduce, for any independent initial $x$, the recursively defined sequence $(\x_{x,m})_{m\in \mathbb{N}_{0}}$, by $\x_{x,0}:=x$ and $\x_{x,m}:=T(\beta_{m})\x_{x,m-1}+\eta_{m}$ for all $m \in \mathbb{N}$. Now, introduce $\X_{x}:[0,\infty)\times \Omega \rightarrow V$ by $\X_{x}(t):= \sum \limits_{m=0}\limits^{\infty}T((t-\alpha_{m})_{+})(\x_{x,m})\id_{[\alpha_{m},\alpha_{m+1})}(t)$ for all $t \geq 0$, where $\alpha_{0}:=0$ and $\alpha_{m}:=\sum \limits_{k=1}\limits^{m}\beta_{k}$. Finally, set $N(t):=\sum \limits_{m=0}\limits^{\infty}m\id\{ \alpha_{m}\leq t<\alpha_{m+1} \}$, then $(N(t))_{t \geq 0}$ is a Poisson process and we have $\X_{x}(t)=T(t-\alpha_{N(t)})\x_{x,N(t)}$, for all $t \geq 0$, almost surely.\\

Processes like $(\X_{x}(t))_{t \geq 0}$ arise as solutions of abstract Cauchy problems driven by Poisson processes, see \cite{ich2} for more details; particularly \cite[Remark 3.14]{ich2}. Moreover, there is also an intuitive interpretation of the phenomena that are modeled by $(\X_{x}(t))_{t \geq 0}$: Assume $(T(t))_{t \geq 0}$ describes the time-change of some physical process, which is a reasonable assumption, since nonlinear semigroups arise naturally as solutions of evolution equations. Now assume that this physical process is at each of the succeeding times $\alpha_{m}$ exposed to the shock $\eta_{m}$. Then, the process describing the shocked system is $(\X_{x}(t))_{t\geq 0}$. More concretely: The (solution of the) $p$-Laplacian equation we consider as an example, can be used to model the evolution of a hill that consists mostly out of sand, see \cite{birnirtheory}. In this case $T(t)u$ describes the hill's surface at time $t$, $u$ is the hill's initial shape and the $\eta_{k}$'s then could be rain showers, or storms, etc.\\

Now, let us describe this paper's highlights as well as the techniques employed to prove them: Firstly, $(\X_{x}(t))_{t \geq 0}$ is (w.r.t. the completion of its natural filtration and any independent initial $x$) a time-continuous Markov process. For proving this, it is crucial that $(\beta_{m})_{m \in \mathbb{N}}$ is not any i.i.d. sequence, but one consisting of exponentially distributed random variables. Moreover, due to the contractivity and time-continuity of $(T(t))_{t \geq 0}$, the transition semigroup of $(\X_{x}(t))_{t \geq 0}$ has the e-property and the Feller property.\\
For proving these results, we only need the assumptions that have been stated in this introduction so far. But, obtaining more sophisticated results requires the following decay assumption on $(T(t))_{t \geq 0}$: There is a w.r.t. $(T(t))_{t\geq 0}$ invariant, separable and dense sub-Banach space $(W,||\cdot||_{W})\subseteq V$, with continuous injection, such that there are constants $\kappa,\rho \in (0,\infty)$ such that 
\begin{align}
\label{intro_eqbound}
||T(t)w_{1}-T(t)w_{2}||_{W}\leq \left(\kappa t+||w_{1}-w_{2}||_{W}^{-\frac{1}{\rho}}\right)^{-\rho},~\forall t \in [0,\infty)
\end{align}
and $w_{1},w_{2}\in W$. Moreover, we have to assume that $||\eta_{k}||_{V}\in L^{2}(\Omega)$ and that $T(t)0=0$ for all $t \in [0,\infty)$. The latter is due to the nonlinearity indeed not necessarily true, but it is "usually" easily verified whether it holds.\\
As we shall see, (\ref{intro_eqbound}) enables us to derive upper bounds for $||\X_{x}(t)||_{V}$ and $||\X_{x}(t)-\X_{x}(t)||_{V}$. These bounds, together with the e-property allow us to conclude by the aid of the results in \cite{Szarek}, that the transition function of $(\X_{x}(t))_{t \geq 0}$ possesses a unique invariant probability measure $\bar{\mu}:\B(V)\rightarrow [0,1]$. From there, we infer that
\begin{align}\tag{SLLN}
\label{intro_slln}
\lim \limits_{t \rightarrow \infty } \frac{1}{t} \int \limits_{0} \limits^{t} \psi (\X_{x}(\tau)) d\tau =\overline{(\psi)}:= \int \limits_{V}\psi(v)\bar{\mu}(dv),
\end{align}
with probability one, for any Lipschitz continuous $\psi:V \rightarrow \R$ and any independent initial $x$. Once this is achieved we will employ the results in \cite{Holzmann} to prove that: If, in addition the constant $\rho$ appearing in (\ref{intro_eqbound}) fulfills $\rho > \frac{1}{2}$, then there is a $\sigma^{2}(\psi) \in [0,\infty)$ such that
\begin{align}\tag{CLT}
\label{intro_clt}
\lim \limits_{t \rightarrow \infty }\frac{1}{\sqrt{t}}\left( \int \limits_{0}\limits^{t}\psi(\X_{x}(\tau))d\tau-t\overline{(\psi)}\right)=Y\sim N(0,\sigma^{2}(\psi)) ,
\end{align}
in distribution, for any Lipschitz continuous $\psi:V \rightarrow \R$ and any independent initial $x$.\\
Finally, we will illustrate the applicability of these results with two examples: In the first one $(T(t))_{t \geq 0}$ is the semigroup of solutions of a first order, nonlinear ODE. Consequently, in this case $(T(t))_{t \geq0}$ is a semigroup on $\R$. By aid of this example, we will demonstrate that (\ref{intro_clt}) can fail if (\ref{intro_eqbound}) only holds for a $\rho\in (0,\frac{1}{2}]$; particularly, even if $\rho=\frac{1}{2}$.\\ 
In our second (more sophisticated) example, $(T(t))_{t\geq 0}$ acts on an infinite dimensional Banach space and is the semigroup of strong solutions of the weighted $p$-Laplacian evolution equation with Neumann boundary conditions for large $p$, i.e. $p \in (2,\infty)$. We will prove that in this case, (\ref{intro_slln}) holds for any $p \in (2,\infty)$ and (\ref{intro_clt}) holds if $p \in (2,4)$.\\

Results analogous to (\ref{intro_slln}) and (\ref{intro_clt}) are proven in \cite{ich2}. But there it is assumed that the involved semigroup fulfills a finite extinction assumption and not a polynomial decay assumption. Polynomial decay and finite extinction are probably the most common asymptotic behaviors exhibited by nonlinear semigroups. Even though the i.i.d.-splitting method employed in \cite{ich2} and the Markov process technique used in this paper have essentially nothing in common, one can consider these two papers as natural complements of each other. Particularly, the example considered in \cite{ich2} is the $p$-Laplacian evolution equation for ''small'' $p$.\\
The results proven in the current paper mainly rely on \cite{Szarek} and \cite{Holzmann}. Of course there are many other general criteria dealing with ergodicity as well as the SLLN and the CLT for Markov processes. Particularly interesting criteria can be found in the book \cite{Kulik}.\\

Finally, let us briefly outline this paper's structure: Section \ref{sec_notaprel} clarifies this paper's notation, states some basic results with a focus on nonlinear semigroups and concludes with some elementary properties of $(\X_{x}(t))_{t\geq 0}$ - for technical conveniences the results in Section \ref{sec_notaprel} are formulated without any distributional assumptions on the involved random variables. We then proceed in Section \ref{sec_mp} by proving that $(\X_{x}(t))_{t\geq 0}$ is a time-homogeneous Markov process and demonstrate that it possesses, among others, the Feller and the e-property. Section \ref{sec_sllnclt} is this section's centerpiece, it is proven there that the transition function of $(\X_{x}(t))_{t\geq 0}$ possesses a unique invariant probability measure (Proposition \ref{prop_uniqueinvpropmeas}) that it fulfills the SLLN (Theorem \ref{theorem_slln}) as well as the CLT (Theorem \ref{theorem_clt}). Finally, in Section \ref{sec_examples} we start with a general differential inequality result useful to prove (\ref{intro_eqbound}), then consider the nonlinear ODE example (Remark \ref{remarkex1}) and devote the remainder of this section to the $p$-Laplacian example.

\section{Notation and preliminary Results}
\label{sec_notaprel}

This section starts with some remarks regarding nonlinear semigroups, proceeds with some general words on the functional analytic and probability theoretic notations used throughout this paper and concludes with some results regarding the stochastic processes considered in this paper.\\ 
Throughout this section $(V,||\cdot||_{V})$ denotes a separable (real) Banach space and $\B(V)$ its Borel $\sigma$-Algebra.\\

A family of mappings $(T(t))_{t \geq 0}$, where $T(t):V\rightarrow V$ is called a semigroup on $V$, if $T(0)v=v$ and $T(t+h)v=T(t)T(h)v$ for all $t,h \in [0,\infty)$ and $v \in V$. A semigroup $(T(t))_{t \geq 0}$ on $V$ is called
\begin{enumerate}
	\item time-continuous, if $[0,\infty) \ni t \mapsto T(t)v$ is a continuous map for all $v \in V$,
	\item contractive, if $||T(t)v_{1}-T(t)v_{2}||_{V}\leq ||v_{1}-v_{2}||_{V}$ for all $t \in [0,\infty)$ and $v_{1},~v_{2}\in V$,
	\item linear, if the mapping $T(t):V \rightarrow V$ is linear for all $t \in [0,\infty)$.
\end{enumerate}

Particularly in the linear theory, one frequently uses the phrase "strongly continuous" instead of "time continuous" semigroup and implicitly assumes that a (linear) strongly continuous semigroup consists of linear and continuous operators. Of course, using our notation a linear, time-continuous semigroup does not necessarily consist of continuous operators. Therefore, to avoid confusion, we chose to use "time continuous" instead of "strongly continuous."

\begin{remark}\label{remark_measuc0sg} Let $(T(t))_{t \geq 0}$ be a time-continuous and contractive semigroup on $V$. Then it is easily verified that $T$ is also jointly continuous, i.e.  $[0,\infty) \times V \ni (t,v)\mapsto T(t)v$ is a continuous map. Consequently, this map is a fortiori $\mathfrak{B}([0,\infty)\times V)$-$\mathfrak{B}(V)$-measurable. Moreover, by separability we have $\mathfrak{B}([0,\infty)\times V)= \mathfrak{B}([0,\infty))\otimes \mathfrak{B}(V)$, see \cite[page 244]{Billingsley}; which gives that this map is $\mathfrak{B}([0,\infty))\otimes \mathfrak{B}( V)$-$\mathfrak{B}(V)$-measurable.
\end{remark}

The following remark gives the connection between nonlinear semigroups and nonlinear evolution equations. It is not needed for the Sections \ref{sec_mp} and \ref{sec_sllnclt}, but for the $p$-Laplacian example considered in Section \ref{sec_examples}. It is stated now, as it reveals the significance of nonlinear semigroups and therefore motivates why we consider them.

\begin{remark}\label{remark_msee} A mapping $\mathcal{A}:V\rightarrow 2^{V}$ is called multi-valued operator and $D(\mathcal{A}):=\{v\in V: \mathcal{A}v\neq \emptyset\}$ is called its domain. $\A$ is single-valued if $\mathcal{A}v$ contains precisely one element for all $v\in D(\mathcal{A})$. Moreover, instead of $\mathcal{A}:V\rightarrow 2^{V}$ we may write $\mathcal{A}:D(\mathcal{A})\rightarrow 2^{V}$. In addition, by identifying $\A$ with its graph $G(\A):=\{(v,\hat{v}):~v \in D(\A),~\hat{v}\in \A v\}$ we may write $(v,\hat{v})\in \A$ instead of $v \in D(\A)$ an $\hat{v}\in \A v$. Furthermore, $\A:D(\A)\rightarrow 2^{V}$ is called accretive, if $||v_{1}-v_{2}||_{V}\leq ||v_{1}-v_{2}+\alpha(\hat{v}_{1}-\hat{v}_{2})||_{V}$ for all $\alpha>0$, $(v_{1},\hat{v}_{1}),(v_{2},\hat{v}_{2}) \in \A$; m-accretive, if it is accretive and $Range(Id+\alpha\A)=V$, for all $\alpha>0$; and densely defined if $\overline{D(\A)}=V$.\\
Moreover, we have the following celebrated result connection nonlinear semigroups and evolution equations: Let $\mathcal{A}:V\rightarrow 2^{V}$ be densely defined and m-accretive. Then, the initial value problem
	\begin{align}
	\label{remark_eqms}
	0 \in u^{\prime}(t)+\mathcal{A}u(t),~\text{for a.e. }t\in (0,\infty),~u(0)=v,
	\end{align}
has precisely one mild solution, see \cite[Definition 1.3]{BenilanBook} and \cite[Prop. 3.7]{BenilanBook}. Moreover, the family of mappings $(T_{\A}(t))_{t \geq 0}$ such that $T_{\A}(\cdot)v$ is, for each $v \in V$, the mild solution of (\ref{remark_eqms}) forms  a time-continuous, contractive semigroup on $V$, see \cite[Theorems 1.10 and 3.10]{BenilanBook}, and will be called "the semigroup associated to $\A$".
\end{remark}

The reader is referred to \cite{BenilanBook} for a comprehensive introduction to nonlinear semigroups. Moreover, the book \cite{acmbook} deals with existence, uniqueness, asymptotic and qualitative results for numerous evolution equations and this book's appendix contains a more concise introduction to this topic.\\
Even though we also consider linear semigroups, namely the semigroup associated to a Markov process, no profound knowledge of linear semigroups is required to understand this paper.\\  

Given a measure space $(K,\Sigma,\nu)$, we denote by $L^{q}(K,\Sigma,\nu)$, where $q \in [1,\infty]$, the usual Lebesgue spaces of ($\nu$-equivalence classes of) real-valued, $\Sigma$-$\B(\mathbb{R})$-measurable functions $f:K \rightarrow \mathbb{R}$, such that: $|f|^{q}$ is Lebesgue integrable, if $q \neq \infty$; $\nu$-essentially bounded, if $q= \infty$.\\
Moreover, we introduce the spaces $\text{BM}(V)$, $C_{b}(V)$, $\text{Lip}_{b}(V)$ and $\text{Lip}(V)$ as the spaces of all functions $\psi:V\rightarrow \mathbb{R}$ which are bounded and measurable, continuous and bounded, Lipschitz continuous and bounded, and Lipschitz continuous, respectively. Moreover, for any Lipschitz continuous function $\psi$, we denote its Lipschitz constant by $L_{\psi}$.\\

Throughout everything which follows $(\Omega,\F,\P)$ denotes a complete probability space. Moreover, we introduce the short cut notation $L^{q}(\Omega,\F,\P):=L^{q}(\Omega)$ for all $q \in [1,\infty)$. In addition, $\mathcal{M}(\Omega;V)$ denotes the space of $V$-valued random variables, i.e. all $\F$-$\B(V)$-measurable mappings $Y:\Omega \rightarrow V$.\\ 
If $Y_{i}$ is a $V_{i}$-valued random variable for each $i \in I$, where $I$ is an arbitrary index set and the $V_{i}$'s are separable Banach spaces, then $\sigma(Y_{j};j \in I)\subseteq \F$ is the smallest  $\sigma$-Algebra, such that each $Y_{i}$ is $\sigma(Y_{j};j \in I)-\B(V_{i})$-measurable. In addition, $\sigma_{0}(Y_{j};j \in I)$ denotes its completion, i.e.
\begin{align*}
\sigma_{0}(Y_{j};j \in I):=\{A \in \F:~ \exists B \in \sigma(Y_{j};j \in I)\text{, such that } \P(A \Delta B)=0 \},
\end{align*}
where $\Delta$ denotes the symmetric difference. It is easily verified that the right-hand-side of the previous equation is indeed a $\sigma$-Algebra and the smallest one containing all $\P$-null-sets as well as all elements of $\sigma(Y_{j};j \in I)$. Moreover, it is well known that an $Y \in \mathcal{M}(\Omega;V)$ is independent of a $\sigma$-algebra, if and only if it is independent of the $\sigma$-algebra's completion.\\
Finally, for any $Y \in \mathcal{M}(\Omega;V)$, we denote by $\P_{Y}$ its law, i.e. $\P_{Y}(B):=\P(Y\in B)$, for all $B \in \B(V)$.\\
Even though we mostly consider real-valued functionals of vector-valued processes, some results on random variables (and stochastic processes) taking values in Banach spaces are needed in the sequel. For a concise introduction, see \cite[Chapter 2]{SIBS}.\\

\noindent Now, let us spend some words on the stochastic process which is the central object of this paper:

\begin{definition}\label{def_proc} Let $(\beta_{m})_{m \in \mathbb{N}}$, where $\beta_{m}:\Omega \rightarrow (0,\infty)$, be a sequence of real-valued random variables. Moreover, let $(\eta_{m})_{m \in \mathbb{N}}\subseteq\mathcal{M}(\Omega;V)$, introduce $\alpha_{m}:=\sum \limits_{k=1}\limits^{m}\beta_{k}$ for all $m \in \mathbb{N}$ and set $\alpha_{0}:=0$. Finally, let $x \in \mathcal{M}(\Omega;V)$ and let $(T(t))_{t \geq 0}$ be a time-continuous, contractive semigroup on $V$. Then the sequence $(\x_{x,m})_{m\in \mathbb{N}_{0}}$ defined by $\x_{x,0}:=x$ and
\begin{align*}
\x_{x,m}:=T(\alpha_{m}-\alpha_{m-1})\x_{x,m-1}+\eta_{m}=T(\beta_{m})\x_{x,m-1}+\eta_{m},~\forall m \in \mathbb{N},
\end{align*}
is called the sequence generated by $((\beta_{m})_{m \in \mathbb{N}},(\eta_{m})_{m \in \mathbb{N}},x,T)$ in $V$. Moreover, the stochastic process $\X_{x}:[0,\infty)\times \Omega \rightarrow V$ defined by
\begin{align*}
\X_{x}(t):= \sum \limits_{m=0}\limits^{\infty}T((t-\alpha_{m})_{+})(\x_{x,m})\id_{[\alpha_{m},\alpha_{m+1})}(t),~\forall t\geq 0,
\end{align*}
is called the process generated by $((\beta_{m})_{m \in \mathbb{N}},(\eta_{m})_{m \in \mathbb{N}},x,T)$ in $V$.
\end{definition}

\begin{remark}\label{remark_Xmeascad} Let $(\x_{x,m})_{m\in \mathbb{N}_{0}}$ and $\X_{x}:[0,\infty)\times \Omega \rightarrow V$ be the sequence and the process generated by some $((\beta_{m})_{m \in \mathbb{N}},(\eta_{m})_{m \in \mathbb{N}},x,T)$ in $V$. Then it follows easily from Remark \ref{remark_measuc0sg} that each $\x_{x,m}$ and each $\X_{x}(t)$ is $\F$-$\B(V)$-measurable.
\end{remark}

Let us conclude this section with the following lemma, which reveals that the stochastic process generated by some $((\beta_{m})_{m \in \mathbb{N}},(\eta_{m})_{m \in \mathbb{N}},x,T)$ in $V$, depends continuously on $(\eta_{m})_{m \in \mathbb{N}}$ and $x$.

\begin{lemma}\label{lemma_xcont} Let $(\beta_{m})_{m \in \mathbb{N}}$, where $\beta_{m}:\Omega \rightarrow (0,\infty)$, be a sequence of real-valued random variables. Moreover, let $(\eta_{m})_{m \in \mathbb{N}},~(\hat{\eta}_{m})_{m \in \mathbb{N}}\subseteq\mathcal{M}(\Omega;V)$ and $x,~\hat{x} \in \mathcal{M}(\Omega;V)$. In addition, introduce $\alpha_{m}:=\sum \limits_{k=1}\limits^{m}\beta_{k}$ for all $m \in \mathbb{N}$, set $\alpha_{0}:=0$ and define $N(t):= \sum \limits_{m=0}\limits^{\infty}m\id\{ \alpha_{m}\leq t<\alpha_{m+1} \}$ for all $t \in [0,\infty)$. Finally, let $\X_{x}$ and $\hat{\X}_{\hat{x}}$ be the processes generated by $((\beta_{m})_{m \in \mathbb{N}},(\eta_{m})_{m \in \mathbb{N}},x,T)$ in $V$ and $((\beta_{m})_{m \in \mathbb{N}},(\hat{\eta}_{m})_{m \in \mathbb{N}},\hat{x},T)$ in $V$, respectively; where $(T(t))_{t \geq 0}$ is a time-continuous, contractive semigroup on $V$. Then the assertion
\begin{align}
\label{lemma_twoprocessboundeq}
||\X_{x}(t)-\hat{\X}_{\hat{x}}(t)||_{V}\leq ||x-\hat{x}||_{V}+\sum \limits_{k=1}\limits^{N(t)}||\eta_{k}-\hat{\eta}_{k}||_{V},~\forall t \in [0,\infty)
\end{align}
holds on $\Omega$.
\end{lemma}
\begin{proof} Let $t \in [0,\infty)$ be given and set $M_{t}:=\{\omega \in \Omega: t< \sup \limits_{m \in \mathbb{N}} \alpha_{m}(\omega) \}$. Note that $\X_{x}(t)=\hat{\X}_{\hat{x}}(t)=0$ on $\Omega \setminus M_{t}$, thus (\ref{lemma_twoprocessboundeq}) holds on $\Omega \setminus M_{t}$. Moreover, on $M_{t}$ we have by contractivity of $(T(t))_{t \geq 0}$ that
\begin{align*}
||\X_{x}(t)-\hat{\X}_{\hat{x}}(t)||_{V} = ||T(t-\alpha_{N(t)})\x_{x,N(t)}-T(t-\alpha_{N(t)})\hat{\x}_{\hat{x},N(t)}||_{V}\leq ||\x_{x,N(t)}-\hat{\x}_{\hat{x},N(t)}||_{V},
\end{align*}
where $(\x_{x,m})_{m\in \mathbb{N}_{0}}$ and $(\hat{\x}_{\hat{x},m})_{m\in \mathbb{N}_{0}}$ denote the sequences generated by $((\beta_{m})_{m \in \mathbb{N}},(\eta_{m})_{m \in \mathbb{N}},x,T)$ in $V$ and $((\beta_{m})_{m \in \mathbb{N}},(\hat{\eta}_{m})_{m \in \mathbb{N}},\hat{x},T)$ in $V$, respectively. Consequently, it suffices to prove that
\begin{align*}
||\x_{x,m}-\hat{\x}_{\hat{x},m}||_{V} \leq ||x-\hat{x}||_{V}+\sum \limits_{k=1}\limits^{m}||\eta_{k}-\hat{\eta}_{k}||_{V},~\forall m \in \mathbb{N}_{0}.
\end{align*}
For $m=0$ this is clear and for $m \in \mathbb{N}_{0}$ we get $||\x_{x,m+1}-\hat{\x}_{,\hat{x},m+1}||_{V} \leq ||\x_{x,m}-\hat{\x}_{\hat{x},m}||_{V}+||\eta_{m+1}-\hat{\eta}_{m+1}||_{V}$, which yields the claim by induction.
\end{proof}

\section{The Markov Property}
\label{sec_mp}

Throughout this section, $(V,||\cdot||_{V})$ is a separable Banach space and $(\eta_{m})_{m \in \mathbb{N}} \subseteq \mathcal{M}(\Omega;V)$ denotes an i.i.d. sequence. In addition $(\beta_{m})_{m \in \mathbb{N}}$, where $\beta_{m}:\Omega \rightarrow (0,\infty)$, is an i.i.d. sequence of exponentially distributed random variables, with parameter $\theta \in (0,\infty)$. Furthermore, we assume that $(\eta_{m})_{m \in \mathbb{N}}$ and $(\beta_{m})_{m \in \mathbb{N}}$ are independent of each other.\\   
Now, set $\alpha_{m}:= \sum \limits_{k=1}\limits^{m}\beta_{k}$ for all $m \in \mathbb{N}$, introduce $\alpha_{0}:=0$ and $N: [0,\infty)\times \Omega \rightarrow \mathbb{N}_{0}$ by
\begin{align}
\label{eq_poissonp}
N(t):= \sum \limits_{m=0}\limits^{\infty}m\id\{ \alpha_{m}\leq t<\alpha_{m+1} \},~\forall t \in [0,\infty).
\end{align}
Moreover, let $(T(t))_{t \geq 0}$ be a time-continuous, contractive semigroup on $V$.\\

An $x \in \mathcal{M}(\Omega;V)$ is called an independent initial, if $x$ is jointly independent of $(\eta_{m})_{m \in \mathbb{N}}$ and $(\beta_{m})_{m \in \mathbb{N}}$. Moreover, for any independent initial $x \in \mathcal{M}(\Omega;V)$ we denote by $(\x_{x,m})_{m\in \mathbb{N}_{0}}$ the sequence and by $\X_{x}:[0,\infty)\times \Omega \rightarrow V$ the process generated by $((\beta_{m})_{m \in \mathbb{N}},(\eta_{m})_{m \in \mathbb{N}},x,T)$ in $V$. Moreover, we may identify a $v \in V$ with the random variable which is constantly $v$ and note that any $v \in V$ is obviously an independent initial. In addition,  we introduce the filtration $(\F^{x}_{t})_{t \geq 0}$, by $\F^{x}_{t}:= \sigma_{0}(\X_{x}(\tau);~\tau \in [0,t])$, for all $t \in [0,\infty)$ and any independent initial $x \in \mathcal{M}(\Omega;V)$. Finally, let $P:[0,\infty)\times V \times \B(V)\rightarrow [0,1]$ be defined by
\begin{align}
\label{eq_pdef}
P(t,v,B):=\P(\X_{v}(t)\in B)=\P_{X_{v}(t)}(B),
\end{align}
for all $v \in V$, $t \in [0,\infty)$ and $B \in \B(V)$.\\

The purpose of this section is to show that $\X_{x}$ is for any independent initial $x \in \mathcal{M}(\Omega;V)$ a time homogeneous Markov process with transition function $P$ and initial distribution $\P_{x}$. In addition, we will establish some basic properties of these quantities.\\
Before embarking on theses tasks, let us clarify the following: Throughout this entire section, we do not assume that $(T(t))_{t \geq 0}$ exhibits the introductory mentioned (or any other) decay behavior. Consequently, one can also apply this section's results under possibly different decay assumptions on the involved semigroup.

\begin{remark}\label{remark_xrepresentation} Let  $x \in \mathcal{M}(\Omega;V)$ be an independent initial. Appealing to the strong law of large numbers yields $\lim \limits_{m \rightarrow \infty}\alpha_{m}=\infty$ almost surely. Consequently, on a set $\tilde{\Omega}\in \F$ of full $\P$-measure, we can find for each $\omega \in \tilde{\Omega}$ and $t \in [0,\infty)$  precisely one $m \in \mathbb{N}$, s.t. $t \in [\alpha_{m}(\omega),\alpha_{m+1}(\omega))$. Thus, we get
\begin{align*}
\P \left(\X_{x}(t)=T(t-\alpha_{N(t)})\x_{x,N(t)},~\forall t \geq 0\right)=1.
\end{align*} 
In addition, it is well known that $(N(t))_{t\geq 0}$ is a homogeneous Poisson process with intensity $\theta$.
\end{remark}

\begin{lemma}\label{lemma_meas} The mapping defined by $[0,\infty)\times V \times \Omega \ni (t,v,\omega)\mapsto \X_{v}(t,\omega)$ is $\B([0,\infty))\otimes \B(V)\otimes \F$-$\B(V)$-measurable. Consequently, if $\psi\in \text{BM}(V)$, then $[0,\infty)\times V \ni (t,v) \mapsto \mathbb{E}\psi(\X_{v}(t))$ is $\B([0,\infty))\otimes \B(V)$-$\B(\R)$-measurable
\end{lemma}
\begin{proof} For the moment, let $v \in V$ and $\omega \in \Omega$ be fixed. If, for a given $t \in [0,\infty)$, we have $t \geq \sup \limits_{m \in \mathbb{N}}\alpha_{m}(\omega)$, then the same holds for $t+h$, for any $h\geq 0$. Thus, we get $\X_{v}(t+h,\omega)=\X_{v}(t,\omega)=0$. Moreover, if $t < \sup \limits_{m \in \mathbb{N}}\alpha_{m}(\omega)$, we have $N(t+h,\omega)=N(t,\omega)$ for all $h\geq 0$ sufficiently small. Consequently, the time-continuity of $T$ yields that $[0,\infty)\ni t \mapsto \X_{v}(t,\omega)$ is right-continuous for all $v \in V$ and $\omega \in \Omega$. Consequently, as each $\X_{v}(t)$ is $\F$-$\B(V)$-measurable we get that $\X_{v}$ is $\B([0,\infty))\otimes \F$-$\B(V)$-measurable, see \cite[Prop. 2.2.3.ii]{SIBS}. In addition, appealing to Lemma \ref{lemma_xcont} yields that $V \ni v\mapsto \X_{v}(t,\omega)$ is continuous for all $t \in [0,\infty)$ and $\omega \in \Omega$. Consequently, $[0,\infty)\times V \times \Omega \ni (t,v,\omega)\mapsto \X_{v}(t,\omega)$ is $ \B([0,\infty))\otimes \B(V)\otimes \F$-$\B(V)$-measurable, by \cite[Lemma 4.51]{IDA}.\\ 
Finally, let $\psi \in \text{BM}(V)$, then the boundedness of $\psi$ yields that the expectation at hand exists; and the already proven measurability result (together with \cite[Prop. 2.1.4]{SIBS}) enables us to conclude the remaining claim.
\end{proof}

\begin{theorem}\label{theorem_mp} Let $x \in \mathcal{M}(\Omega;V)$ be an independent initial. Then $(\X_{x}(t))_{t \geq 0}$ is a Markov process with respect to $(\F_{t}^{x})_{t \geq 0}$, i.e.
\begin{align}
\label{theorem_mpeq1}
\P(\X_{x}(t+h)\in B|\F_{t}^{x})=\P(\X_{x}(t+h)\in B|\X_{x}(t)),
\end{align}
almost surely, for all $t,h \in [0,\infty)$ and $B \in \B(V)$. In addition, $P$ is a time homogeneous transition function, that is
\begin{enumerate}[i)]
	\item\label{theorem_mpenumi1} $\B(V)\ni B \mapsto P(t,v,B)$ is a probability measure on $(V,\B(V))$, for all $t \in [0,\infty)$ and $v \in V$,
	\item\label{theorem_mpenumi2} $\P(0,v,B)=\id_{B}(v)$ for all $v \in V$ and $B \in \B(V)$, 
	\item\label{theorem_mpenumi3}  $P(\cdot,\cdot,B)$ is $\B([0,\infty))\otimes \B(V)$-$\B([0,1])$-measurable for any $B \in \B(V)$ and
	\item\label{theorem_mpenumi4} $P$ fulfills has the Chapman-Kolmogorov property, i.e. $P(t+h,v,B)=\int \limits_{V} P(h,\hat{v},B)dP(t,v,d\hat{v})$ for all $t,h \in [0,\infty)$, $v \in V$ and $B \in \B(V)$.
\end{enumerate}
Moreover, $(\X_{x}(t))_{t \geq 0}$ is time homogeneous (with initial distribution $\P_{x}$) and transition function $P$, i.e.
\begin{align}
\label{theorem_mpeq2}
\P(\X_{x}(t+h)\in B|\F_{t}^{x})=P(h,\X_{x}(t),B),
\end{align} 
almost surely, for all $t,h \in [0,\infty)$ and $B \in \B(V)$.
\end{theorem}
\begin{proof} The assertions \ref{theorem_mpenumi1}) and \ref{theorem_mpenumi2}) are trivial. Moreover, the third follows from Lemma \ref{lemma_meas}.\\
Proving the remaining assertions is more involved and will occupy us for some time. Let us start with some preparatory observations. To this end, let $t,h \in [0,\infty)$, $v \in V$ and $B \in \B(V)$ be given; and introduce $F_{m}:V\times [0,\infty)^{m}\times V^{m}\rightarrow V$, for all $m \in \mathbb{N}$, by $F_{1}(y,b,n):=T(b)y+n$ and $F_{m}(y,b_{1},..,b_{m},n_{1},..,n_{m}):=T(b_{m})F_{m-1}(y,b_{1},..,b_{m-1},n_{1},..,n_{m-1})+n_{m}$ for all $y,n,n_{1},..,n_{m}\in V$, $b_{1},..,b_{m}\in [0,\infty)$ and $m \in \mathbb{N}\setminus \{1\}$.\\
Appealing to Remark \ref{remark_measuc0sg} yields that $F_{1}$ is continuous and it then follows inductively that each $F_{m}$ has this property and is therefore $\B(V)\otimes \B([0,\infty)^{m})\otimes \B(V^{m})$-$\B(V)$-measurable.\\
Now, for the sake of space let $\hat{\eta}_{\tau,m}:=(\eta_{N(\tau)+1},..,\eta_{N(\tau)+m})$, for all $m \in \mathbb{N}$, $\tau \in [0,\infty)$ and\linebreak $\hat{\beta}_{\tau,m}:=(\alpha_{N(\tau)+1}-\tau,\beta_{N(\tau)+2},..,\beta_{N(\tau)+m})$ if $m \geq 2$ and $\hat{\beta}_{\tau,1}:=\alpha_{N(\tau)+1}-\tau$ for all $\tau \in [0,\infty)$ and let us prove inductively that
\begin{align}
\label{theorem_mpprroof1}
\x_{x,N(\tau)+m}=F_{m}(\X_{x}(\tau),\hat{\beta}_{\tau,m},\hat{\eta}_{\tau,m}),~\forall \tau \in [0,\infty),
\end{align}
almost surely for all $m \in \mathbb{N}$.\\ 
If $m=1$, we get by the semigroup property and Remark \ref{remark_xrepresentation} that 
\begin{align*}
\x_{x,N(\tau)+1}=T(\alpha_{N(\tau)+1}-\tau)\X_{x}(\tau)+\eta_{N(\tau)+1}=F_{1}(\X_{x}(\tau),\alpha_{N(\tau)+1}-\tau,\eta_{N(\tau)+1})
\end{align*}
almost surely. Moreover, if (\ref{theorem_mpprroof1}) holds for an $m \in \mathbb{N}$ we get
\begin{eqnarray*}
\x_{x,N(\tau)+m+1}
& = & ~ T(\beta_{N(\tau)+m+1})\x_{x,N(\tau)+m}+\eta_{N(\tau)+m+1} \\
& = & ~ T(\beta_{N(\tau)+m+1})F_{m}(\X_{x}(\tau),\hat{\beta}_{\tau,m},\hat{\eta}_{\tau,m})+\eta_{N(\tau)+m+1} \\
& = & ~ F_{m+1}(\X_{x}(\tau),\hat{\beta}_{\tau,m},\beta_{N(\tau)+m+1},\hat{\eta}_{\tau,m},\eta_{N(\tau)+m+1}),
\end{eqnarray*}
which yields (\ref{theorem_mpprroof1}). Consequently, on $\{N(\tau+h)=N(\tau)\}$, we have
\begin{align}
\label{theorem_mpprroof5}
\X_{x}(\tau+h)=T(\tau+h-\alpha_{N(\tau)})\x_{x,N(\tau)}=T(h)\X_{x}(\tau),~\forall \tau,h\in [0,\infty).
\end{align}
up-to a $\P$-null-set and on $\{N(\tau+h)=N(\tau)+m\}$, where $m \in \mathbb{N}$, we have
\begin{align}
\label{theorem_mpprroof6}
\X_{x}(\tau+h)=T(\tau+h-\alpha_{N(\tau)+m})F_{m}(\X_{x}(\tau),\hat{\beta}_{\tau,m},\hat{\eta}_{\tau,m}),~\forall \tau,h\in [0,\infty),
\end{align}
up-to a $\P$-null-set. These two results will turn out to be useful to prove (\ref{theorem_mpeq1}) and (\ref{theorem_mpeq2}). But before we can do so some distribution results have to be established, namely
\begin{enumerate}[I)]
	\item For all $m \in \mathbb{N}$, we have that $(\alpha_{N(t)+1}-t,..,\alpha_{N(t)+m}-t,\eta_{N(t)+1},..,\eta_{N(t)+m},N(t+h)-N(t))$ is in distribution equal to $(\alpha_{1},..,\alpha_{m},\eta_{1},..,\eta_{m},N(h))$.
	\item For all $m \in \mathbb{N}$, we have that $(\alpha_{N(t)+1},\beta_{N(t)+2},..,\beta_{N(t)+m},\eta_{N(t)+1},..,\eta_{N(t)+m},N(t+h)-N(t))$ is independent of $\F^{x}_{t}$.
\end{enumerate}
Proof of I). Let $z_{1},..,z_{m}\in [0,\infty)$, $B_{1},..,B_{m}\in \B(V)$ and $C \subseteq \mathbb{N}_{0}$. Then, as  $(\eta_{m})_{m \in \mathbb{N}}$ is i.i.d and independent of $(\alpha_{m})_{m\in \mathbb{N}}$ we get
\begin{eqnarray*}
	& & ~
	\P(\alpha_{N(t)+k}-t\leq z_{k},~\eta_{N(t)+k} \in B_{k},~k=1,..,m,~ N(t+h)-N(t)\in C) \\
	& = & ~ \P(\eta_{k} \in B_{k},~k=1,..,m)\sum \limits_{j=0}\limits^{\infty}\P(\alpha_{j+k}-t\leq z_{k},~k=1,..,m,~ N(t+h)-N(t)\in C,~N(t)=j)\\
		& = & ~ \P(\eta_{k} \in B_{k},~k=1,..,m)\P(N(z_{k}+t)-N(t)\geq k,~k=1,..,m,~ N(t+h)-N(t)\in C)
\end{eqnarray*}
where the last equality follows from  
\begin{align}
\label{theorem_mpprroof4}
\{\alpha_{k}\leq \tau\}=\{N(\tau)\geq k\},~\forall \tau \in [0,\infty),~k \in \mathbb{N}_{0},
\end{align}
up to a $\P$-null-set. Since $(N(t))_{t \geq 0}$ is a homogeneous Poisson process, it is now easily verified that the distribution of $(N(z_{1}+t)-N(t),..,N(z_{m}+t)-N(t),N(t+h)-N(t))$ is independent of $t$. Using this and (\ref{theorem_mpprroof4}) yields
\begin{align*}
\P(N(z_{k}+t)-N(t)\geq k,~k=1,..,m,~ N(t+h)-N(t)\in C) = \P(\alpha_{k}\leq z_{k},~k=1,..,m,~ N(h)\in C).
\end{align*}
Combining the preceding two calculations, while having in mind the independence of $(\eta_{m})_{m\in \mathbb{N}}$ and $(\alpha_{m})_{m\in \mathbb{N}}$ gives i). \\
Proof of II). Since $\beta_{N(t)+k}=\alpha_{N(t)+k}-\alpha_{N(t)+k-1}$ for all $k \in \mathbb{N}\setminus\{1\}$, it suffices to prove that $(\alpha_{N(t)+1},..,\alpha_{N(t)+m},\eta_{N(t)+1},..,\eta_{N(t)+m},N(t+h)-N(t))$ is independent of $\F_{t}^{x}$. The latter is obviously true if $(\alpha_{N(t)+1}-t,..,\alpha_{N(t)+m}-t,\eta_{N(t)+1},..,\eta_{N(t)+m},N(t+h)-N(t))$ is independent of $\F_{t}^{x}$.\\
Now introduce $\Sigma_{\tau}:=\sigma(A \cap B:~ A \in \Sigma^{N}_{\tau},~ B \in \sigma_{0}(\eta_{k},~k \in \mathbb{N}_{0}))$, for all $\tau \in [0,\infty)$, where $(\Sigma^{N}_{\tau})_{\tau \geq 0}$ denotes the completion of the natural filtration of $(N(\tau))_{\tau \geq 0}$ and $\eta_{0}:=x$ and let us prove that
\begin{align}
\label{theorem_mpprroof2}
\F^{x}_{t} \subseteq \Sigma_{t} \text{ and } \eta_{N(t)+j} \text{ is } \Sigma_{t}-\B(V)-\text{measurable for all }j \in \mathbb{N}.
\end{align}
The second assertion is clearly true, since
\begin{align*}
\{\eta_{N(t)+j}\in B\}= \bigcup\limits_{k=0}\limits^{\infty}\{\eta_{k+j}\in B,~N(t)=k\} \in \Sigma_{t},~\forall B \in \B(V).
\end{align*}
Now, note that $\Sigma_{t}$ contains by construction every $\P$-null-set. Consequently, the first assertion follows if $\X_{x}(s)$ is $\Sigma_{t}$-$\B(V)$-measurable, for each $s \in [0,t]$, which will be verified now: So let $s \in [0,t]$, then appealing to (\ref{theorem_mpprroof5}) and (\ref{theorem_mpprroof6}) (with $\tau=0$ and $h=s$ there) yields, for a given $B \in \B(V)$, that
\begin{align*}
\{X_{x}(s)\in B\}=\{T(s)x \in B,N(s)=0\}\cup\left( \bigcup\limits_{n=1}\limits^{\infty}\{T(s-\alpha_{n})F_{n}(x,\beta_{1},..,\beta_{n},\eta_{1},..,\eta_{n})\in B,~N(s)=n\}\right),
\end{align*}
up to a $\P$-null-set. It is plain that the first set in the preceding equation is an element of $\Sigma_{t}$. Moreover, for each $k \in \{1,..,n\}$ and $z \in [0,\infty)$ we have $\{\alpha_{k}\leq z,~N(s)=n\}=\{N(z)\geq k, N(s)=n\}$. If $z \leq s$, this set is clearly in $\Sigma^{N}_{s}$ and if $z>s$, we have $N(z)\geq N(s)$, which gives $\{N(z)\geq k, N(s)=n\}=\{N(s)=n\}\in \Sigma^{N}_{s}$; thus in any case $\{\alpha_{k}\leq z,~N(s)=n\}\in \Sigma_{t}$. Consequently, as $\beta_{k}=\alpha_{k}-\alpha_{k-1}$, we obtain that $\beta_{k}\id_{\{N(s)=n\}}$ is $\Sigma_{t}$-$\B([0,\infty))$-measurable, for all $k=1,..,n$ and $n \in \mathbb{N}$.\\
Hence, we get
\begin{eqnarray*}
& & ~
\{T(s-\alpha_{n})F_{n}(x,\beta_{1},..,\beta_{n},\eta_{1},..,\eta_{n})\in B,~N(s)=n\}\\
& = & ~ \{T(s-(\beta_{1}+..+\beta_{n})\id_{\{N(s)=n\}})F_{n}(x,\beta_{1}\id_{\{N(s)=n\}},..,\beta_{n}\id_{\{N(s)=n\}},\eta_{1},..,\eta_{n})\in B,~N(s)=n\},
\end{eqnarray*}
is in $\Sigma_{t}$, for all $n \in \mathbb{N}$ by the measurability of $F_{n}$ and $T$, which concludes the proof of (\ref{theorem_mpprroof2}).\\
Now let $n \in \mathbb{N}$ $z_{1},..,z_{m}\in [0,\infty)$, $B_{1},..B_{m},D_{1},..,D_{n}\in \B(V)$, $C \subseteq \mathbb{N}_{0}$ and $s_{1},..,s_{n}\in [0,t]$. As $(N(\tau))_{\tau \geq 0}$ is a Poisson process, it is clear that $(N(z_{1}+t)-N(t),..,N(z_{m}+t)-N(t),N(t+h)-N(t))$ is independent of  $\Sigma_{t}^{N}$. Consequently, as $\sigma_{0}(\eta_{k},~k \in \mathbb{N}_{0})$ is independent of all $\Sigma^{N}_{\tau}$, for $\tau \in [0,\infty)$, we get that $(N(z_{1}+t)-N(t),..,N(z_{m}+t)-N(t),N(t+h)-N(t))$ is independent of $\Sigma_{t}$ and that this random vector's distribution does not depend on $t$. Hence, employing (\ref{theorem_mpprroof4}) and (\ref{theorem_mpprroof2}) yields
\begin{eqnarray*}
	& & ~
\P(\alpha_{N(t)+k}-t\leq z_{k},\eta_{N(t)+k}\in B_{k},k=1,..,m,~N(t+h)-N(t) \in C,~\X_{x}(s_{j})\in D_{j},j=1,..,n)\\
& = & ~ \P(N(z_{k}+t)-N(t)\geq k,\eta_{N(t)+k}\in B_{k},k=1,..,m,~N(t+h)-N(t) \in C,\X_{x}(s_{j})\in D_{j},j=1,..,n)\\
& = & ~ \P(N(z_{k})\geq k,k=1,..,m,~N(h) \in C)\P(\eta_{N(t)+k}\in B_{k},k=1,..,m,~\X_{x}(s_{j})\in D_{j},j=1,..,n)\\
& = & ~ \P(\alpha_{k}\leq z_{k},~k=1,..,m,N(h) \in C)\P(\eta_{N(t)+k}\in B_{k},k=1,..,m,~\X_{x}(s_{j})\in D_{j},j=1,..,n).
\end{eqnarray*}
Moreover, we have
\begin{eqnarray*}
	& & ~
	\P(\eta_{N(t)+k}\in B_{k},k=1,..,m,\X_{x}(s_{j})\in D_{j},j=1,..,n)\\
	& = & ~ \sum \limits_{i=0}\limits^{\infty}\sum \limits_{i_{1},.,i_{n}=0}\limits^{i}\P(\eta_{i+k}\in B_{k},k=1,..,m,~ N(t)=i,N(s_{j})=i_{j},T(s_{j}-\alpha_{i_{j}})\x_{x,i_{j}}\in D_{j},j=1,..,n)
\end{eqnarray*}
Now it is easily verified that $\x_{x,i}$ is $\sigma(x,\beta_{1},..,\beta_{i},\eta_{1},..,\eta_{i})$ for  any $i \in \mathbb{N}_{0}$. (For $i=0$ this is trivial, for $i \in \mathbb{N}$ this follows from (\ref{theorem_mpprroof1}) by putting $\tau=0$ there.)\\
Consequently, since $i_{1},..,i_{n} \in \{0,..,i\}$ in the sum of the preceding calculation, we get by the imposed independence assumptions
\begin{align*}
\P(\eta_{N(t)+k}\in B_{k},k=1,..,m,\X_{x}(s_{j})\in D_{j},j=1,..,n)=\P(\eta_{k}\in B_{k},k=1,..,m)\P( \X_{x}(s_{j})\in D_{j},j=1,..,n).
\end{align*}
Finally, putting it all together while having in mind (I) yields
\begin{eqnarray*}
	& & ~
	\P(\alpha_{N(t)+k}-t\leq z_{k},\eta_{N(t)+k}\in B_{k},k=1,..,m,~N(t+h)-N(t) \in C,~\X_{x}(s_{j})\in D_{j},j=1,..,n)\\
	& = & ~ \P(\alpha_{k}\leq z_{k},\eta_{k}\in B_{k}~k=1,..,m,N(h) \in C,)\P( \X_{x}(s_{j})\in D_{j},j=1,..,n)\\
	& = & ~ \P(\alpha_{N(t)+k}-t\leq z_{k},\eta_{N(t)+k}\in B_{k}~k=1,..,m,N(t+h)-N(t) \in C,)\P( \X_{x}(s_{j})\in D_{j},j=1,..,n),
\end{eqnarray*}
which proves (II).\\
Now (\ref{theorem_mpeq1}) and (\ref{theorem_mpeq2}) will be deduced from I)-II) as well as (\ref{theorem_mpprroof5}) and (\ref{theorem_mpprroof6}).\\
Firstly, I) enables us to conclude that $(t+h-\alpha_{N(t)+m},\hat{\beta}_{t,m},\hat{\eta}_{t,m},N(t+h)-N(t))$ is in distribution equal to $(h-\alpha_{m},\hat{\beta}_{0,m},\hat{\eta}_{0,m},N(h))$, since $\beta_{N(t)+m}=(\alpha_{N(t)+m}-t)-(\alpha_{N(t)+m-1}-t)$, for all $m \in \mathbb{N}$.\\
Now, thanks to II) and I)  we can apply well known properties of conditional probabilities (cf. \cite[Prop. 1.43]{SOC}) in the following two calculations; where in the first line of the first calculation (\ref{theorem_mpprroof5}) and in the first line of the second one (\ref{theorem_mpprroof6}) is used.
\begin{eqnarray*}
\mathbb{E}(\id_{B}(\X_{x}(t+h))\id_{\{N(t+h)=N(t)\}}|\F_{t}^{x})(\omega)
& = & ~ \mathbb{E}(\id_{B}(T(h)\X_{x}(t))\id_{\{N(t+h)=N(t)\}}|\F_{t}^{x})(\omega)\\
& = &  ~ \int \limits_{\Omega} \id_{B}(T(h)\X_{x}(t,\omega))\id_{\{N(t+h)=N(t)\}}(\tilde{\omega}) \P(d\tilde{\omega})\\
& = &  ~ \int \limits_{\Omega} \id_{B}(T(h)\X_{x}(t,\omega))\id_{\{N(h)=0\}}(\tilde{\omega}) \P(d\tilde{\omega})\\
\end{eqnarray*}
and
\begin{eqnarray*}
	& & ~
	\mathbb{E}(\id_{B}(\X_{x}(t+h))\id_{\{N(t+h)=N(t)+m\}}|\F_{t}^{x})(\omega)\\
	& = & ~\mathbb{E}(\id_{B}(T(t+h-\alpha_{N(t)+m})F_{m}(\X_{x}(t),\hat{\beta}_{t,m},\hat{\eta}_{t,m}))\id_{\{N(t+h)=N(t)+m\}}|\F_{t}^{x})(\omega)\\
	& = &  ~ \int \limits_{\Omega} \id_{B}(T(t+h-\alpha_{N(t)+m}(\tilde{\omega}))F_{m}(\X_{x}(t,\omega),\hat{\beta}_{t,m}(\tilde{\omega}),\hat{\eta}_{t,m}(\tilde{\omega})))\id_{\{N(t+h)=N(t)+m\}}(\tilde{\omega}) \P(d\tilde{\omega})\\ 
	& = &  ~ \int \limits_{\Omega} \id_{B}(T(h-\alpha_{m}(\tilde{\omega}))F_{m}(\X_{x}(t,\omega),\hat{\beta}_{0,m}(\tilde{\omega}),\hat{\eta}_{0,m}(\tilde{\omega})))\id_{\{N(h)=m\}}(\tilde{\omega}) \P(d\tilde{\omega}),
\end{eqnarray*}
for all $m \in \mathbb{N}$ and $\P$-a.e. $\omega \in \Omega$.\\
Moreover, for any $v \in V$, it is easily verified by induction that $\x_{v,m}=F_{m}(v,\hat{\beta}_{0,m},\hat{\eta}_{0,m})$ for all $m \in \mathbb{N}$ a.s. Consequently, we get
\begin{eqnarray*}
P(h,v,B)
& = & ~\sum \limits_{m=0}\limits^{\infty}\mathbb{E}\left(\id_{B}(T(h-\alpha_{m})\x_{v,m})\id_{\{N(h)=m\}}\right)\\
& = & ~\mathbb{E}\left(\id_{B}(T(h)v)\id_{\{N(h)=0\}}\right)+\sum \limits_{m=1}\limits^{\infty}\mathbb{E}\left(\id_{B}(T(h-\alpha_{m})F_{m}(v,\hat{\beta}_{0,m},\hat{\eta}_{0,m}))\id_{\{N(h)=m\}}\right),
\end{eqnarray*}
for all $v \in V$. Hence, combining the preceding three calculations yields
\begin{align*}
\P(X_{x}(t+h)\in B|\F_{t}^{x})(\omega) = \sum \limits_{m=0}\limits^{\infty}\mathbb{E}(\id_{B}(\X_{x}(t+h))\id_{\{N(t+h)=N(t)+m\}}|\F_{t}^{x})(\omega)=P(h,\X_{x}(t,\omega),B)
\end{align*}
for $\P$-a.e. $\omega \in \Omega$, which proves (\ref{theorem_mpeq2}). Consequently, invoking \ref{theorem_mpenumi3}) gives that the random variable \linebreak$\P(X_{x}(t+h)\in B|\F_{t}^{x})$ is (after a possible modification on a $\P$-null-set) $\sigma(\X_{x}(t))$-$\B([0,1])$-measurable, which yields $\P(\X_{x}(t+h)\in B|\F^{x}_{t}) =\mathbb{E} (\P(\X_{x}(t+h)\in B|\F^{x}_{t})|\X_{x}(t))$ a.s., which implies (\ref{theorem_mpeq1}) by the tower property of conditional expectations.\\ 
Finally, as $x$ was arbitrary, (\ref{theorem_mpeq2}) holds for all independent initials, which is well known to imply \ref{theorem_mpenumi4}) - for the sake of completeness: Appealing to (\ref{theorem_mpeq2}) yields
\begin{align*}
P(t+h,v,B)=\mathbb{E}(\P(\X_{v}(t+h)\in B|\F^{v}_{t}))= \mathbb{E}P(h,\X_{v}(t),B)=\int \limits_{V}, P(h,\hat{v},B)P(t,v,d\hat{v}),
\end{align*}
for all $v \in V$, where the equality of the third and the fourth expression follow from the change of measure formula for expectations, which also holds for vector-valued random variables, see \cite[p. 25]{Billingsley}.
\end{proof}

\begin{remark} In the sequel $(Q(t))_{t \geq 0}$, where $Q(t):\text{BM}(V)\rightarrow \text{BM}(V)$, denotes the family of mappings, defined by
\begin{align}
\label{rema_qdefeq}
(Q(t)\psi)(v):= \mathbb{E}\psi(\X_{v}(t))= \int \limits_{V} \psi(\hat{v})P(t,v,d\hat{v}),
\end{align}
for all $\psi\in \text{BM}(V)$, $v \in V$ and $t \in [0,\infty)$.
\end{remark}

Now, this section concludes by deriving some basic properties of our Markov process. Particularly, the e-property established in the following lemma, opens the door to useful results which enable one to conclude that a (transition function of a) Markov process on a polish state space possesses a unique invariant probability measure, see \cite{Szarek} for more details.

\begin{lemma}\label{lemma_basicprop} The family of mappings $(Q(t))_{t \geq 0}$ has the Feller and the e-property, that is
\begin{enumerate}[i)]
	\item \label{lemma_basicpropmpenumi1} Feller Property: $Q(t)\psi \in C_{b}(V)$ for all $\psi \in C_{b}(V)$.
	\item \label{lemma_basicpropmpenumi2} e-property: For all $\psi \in \text{Lip}_{b}(V)$, $v \in V$ and $\varepsilon>0$, there is a $\delta>0$, such that for all $\hat{v}\in V$, with $||\hat{v}-v||_{V}<\delta$, we have $|(Q(t)\psi)(v)-(Q(t)\psi)(\hat{v})|<\varepsilon$ for all $t \geq 0$.
\end{enumerate} 
Moreover, the following assertions hold for any independent initial $x \in \mathcal{M}(\Omega;V)$.
\begin{enumerate}[i)]\setcounter{enumi}{2}
	\item \label{lemma_basicpropmpenumi3} $(\X_{x}(t))_{t\geq 0}$ has almost surely c\`{a}dl\`{a}g paths and is continuous in probability.
	\item \label{lemma_basicpropmpenumi4} The mapping $[0,\infty) \ni t \mapsto \mathbb{E}\psi(\X_{x}(t))$ is continuous, whenever $\psi \in C_{b}(V)$; in particular, $(Q(\cdot)\psi)(v)$ is continuous for all $\psi \in C_{b}(V)$ and $v \in V$.
	\item \label{lemma_basicpropmpenumi5} The filtration $(\F^{x}_{t})_{t\geq 0}$ fulfills the usual conditions, i.e. it is complete and right right-continuous.
	\item \label{lemma_basicpropmpenumi6} The stochastic process $(\X_{x}(t))_{t\geq 0}$ is $(\F^{x}_{t})_{t\geq 0}$-progressive.
\end{enumerate}
\end{lemma}
\begin{proof} The required boundedness in \ref{lemma_basicpropmpenumi1}) is plain and the desired continuity follows from Lemma  \ref{lemma_xcont} and dominated convergence.\\
Proof of \ref{lemma_basicpropmpenumi2}). Let $\psi \in \text{Lip}_{b}(V)$ and assume that it is not constantly zero, since the claim is trivial in this case. Moreover, let $v \in V$ and $\varepsilon>0$ be given and introduce $\delta:= \frac{\varepsilon}{2L_{\psi}}$. Then employing the services of Lemma \ref{lemma_xcont} once more yields 
\begin{align*}
|(Q(t)\psi)(v)-(Q(t)\psi)(\hat{v})| \leq L_{\psi}\mathbb{E}||\X_{v}(t)-\X_{\hat{v}}(t)||_{V}\leq L_{\psi}||v-\hat{v}||_{V} < \varepsilon,~\forall t \geq 0
\end{align*}
for all $\hat{v}\in V$, with $||v-\hat{v}||_{V}<\delta$, which proves \ref{lemma_basicpropmpenumi2}).\\
Proof of \ref{lemma_basicpropmpenumi3}). It follows analogously to the beginning of the proof of Lemma \ref{lemma_meas} that
\begin{align}
\label{lemma_basicpropeq1}
[0,\infty) \ni t \mapsto \X_{x}(t,\omega) \text{ is right continuous for all }\omega \in \Omega.
\end{align}
Moreover, it is easily verified that the left limits exists on the set $\tilde{\Omega}:=\{\omega \in \Omega:~\lim \limits_{m \rightarrow \infty }\alpha_{m}(\omega)=\infty\}$, which is by the strong law of large numbers a set of full $\P$-measure. Consequently, it remains to prove the continuity in probability, which is in light of (\ref{lemma_basicpropeq1}) true, if $\X_{x}$ is left-continuous in probability. So let $t_{0}\in (0,\infty)$, $t \in [0,t_{0}]$ and $\varepsilon>0$ and note that
\begin{align*}
\P(||\X_{x}(t)-\X_{x}(t_{0})||_{V}>\varepsilon)\leq \P(||\X_{x}(t)-\X_{x}(t_{0})||_{V}>\varepsilon,~N(t)=N(t_{0}))+ P(|N(t)-N(t_{0})|\geq 1).
\end{align*}
Moreover, the contractivity of $(T(t))_{t\geq 0}$ yields
\begin{eqnarray*}
|| \X_{x}(t)-\X_{x}(t_{0})||_{V}
& = & ~ ||T(t-\alpha_{N(t_{0})})\x_{x,N(t_{0})}-T(t-\alpha_{N(t_{0})})T(t_{0}-t)\x_{x,N(t_{0})}||_{V}\\
& \leq & ~ ||\x_{x,N(t_{0})}-T(t_{0}-t)\x_{x,N(t_{0})}||_{V}\\
\end{eqnarray*}
on $\{N(t)=N(t_{0})\}$, up-to a $\P$-null-set. Conclusively, as Poisson processes are well-known to be stochastically continuous and as $(T(t))_{t\geq 0}$ is time-continuous, we get
\begin{align*}
\lim \limits_{t \nearrow t_{0}}\P(||\X_{x}(t)-\X_{x}(t_{0})||_{V}>\varepsilon)\leq \lim \limits_{t \nearrow t_{0}} \P(||\x_{x,N(t_{0})}-T(t_{0}-t)\x_{x,N(t_{0})}||_{V}>\varepsilon,~N(t)=N(t_{0})) = 0,
\end{align*}
which proves \ref{lemma_basicpropmpenumi3}).\\ 
Proof of \ref{lemma_basicpropmpenumi4}). Let $(t_{m})_{m \in \mathbb{N}}$ be converging to a given $t \in [0,\infty)$. Then, we get by \ref{lemma_basicpropmpenumi3}) (and by passing to a subsequence if necessary) that $\lim \limits_{m \rightarrow \infty}||\X_{x}(t_{m})-\X_{x}(t)||_{V}=0$ almost surely. Consequently, $\lim \limits_{m \rightarrow \infty} \psi(\X_{x}(t_{m}))=\psi(\X_{x}(t))$ a.s. and by dominated convergence also in $L^{1}(\Omega)$, which gives \ref{lemma_basicpropmpenumi4}).\\
Finally, the desired completeness in \ref{lemma_basicpropmpenumi5}) holds by construction, the right-continuity follows from \cite[Theorem, p. 556]{Doob}, which is indeed applicable due to \ref{lemma_basicpropmpenumi1}) and (\ref{lemma_basicpropeq1}); and \ref{lemma_basicpropmpenumi6}) follows from  \ref{lemma_basicpropmpenumi5}) and (\ref{lemma_basicpropeq1}) by \cite[Prop. 2.2.3]{SIBS}.
\end{proof}

\section{The SLLN and the CLT}
\label{sec_sllnclt}
Let the notations of the previous section prevail, that means: $(V,||\cdot||_{V})$ is a separable Banach space and $((\eta_{m})_{m \in \mathbb{N}},(\beta_{m})_{m \in \mathbb{N}},T)$ is a fixed triplet, where  $(\eta_{m})_{m \in \mathbb{N}}\subseteq \mathcal{M}(\Omega;V)$ is an i.i.d. sequence,  $(\beta_{m})_{m \in \mathbb{N}}$ is an i.i.d. sequence which is independent of $(\eta_{m})_{m \in \mathbb{N}}$ and  each  $\beta_{m}$ is exponentially distributed with parameter $\theta \in (0,\infty)$, and $(T(t))_{t \geq 0}$ is a time-continuous contractive semigroup on $V$. Moreover, $(N(t))_{t \geq 0}$ is the Poisson process arising from $(\beta_{m})_{m \in \mathbb{N}}$ and $(\alpha_{m})_{m \in \mathbb{N}_{0}}$ is the process' sequence of arrival times.\\
Again we refer to an $x \in \mathcal{M}(\Omega;V)$ which is independent of $((\eta_{m})_{m \in \mathbb{N}},(\beta_{m})_{m \in \mathbb{N}})$ as an independent initial, and denote by $\X_{x}$ and $(\x_{x,m})_{m \in \mathbb{N}_{0}}$ the sequence and the process generated by $((\beta_{m})_{m \in \mathbb{N}},(\eta_{m})_{m \in \mathbb{N}},x,T)$ in $V$.\\
Finally, set $P(t,v,B):=\P(\X_{v}(t)\in B)$ and $(Q(t)\psi)(v):= \mathbb{E}\psi(\X_{v}(t))$, for every $v \in V$, $t \in [0,\infty)$, $B \in \B(V)$ and $\psi \in \text{BM}(V)$.\\
In addition, we assume throughout this entire section that
\begin{align}
\label{eq_etal2int}
||\eta_{k}||_{V} \in L^{2}(\Omega),~\forall k \in \mathbb{N},
\end{align}
where $L^{q}(\Omega):=L^{q}(\Omega,\F,\P)$ for every $q \in [1,\infty)$. Moreover, we impose the following assumption regarding $(T(t))_{t \geq 0}$.

\begin{assumption}\label{assumption} There is a separable Banach space $(W,||\cdot||_{W})$, with $ W \subseteq V$, such that the following assertions hold.
\begin{enumerate}[i)]
	\item\label{assumption_enumi1} The injection $W \hookrightarrow V$ is continuous and $W$ is dense in $(V,||\cdot||_{V})$. 
	\item\label{assumption_enumi2} $T$ is invariant with respect to $W$, that is $T(t)w \in W$ for all $w \in W$ and $t \in [0,\infty)$.
	\item\label{assumption_enumi3} There are constants $\kappa,\rho \in (0,\infty)$ such that $||T(t)w_{1}-T(t)w_{2}||_{W}\leq \left(\kappa t+||w_{1}-w_{2}||_{W}^{-\frac{1}{\rho}}\right)^{-\rho}$ for all $w_{1},w_{2}\in W$ and $t \in [0,\infty)$.
	\item \label{assumption_enumi4} $T(t)0=0$ for all $t \in [0,\infty)$.
\end{enumerate}
\end{assumption}

Throughout this entire section, Assumption \ref{assumption} is assumed to hold; particularly, $(W,||\cdot||_{W})$ and $\kappa,\rho \in (0,\infty)$ are such that \ref{assumption}.\ref{assumption_enumi1})-\ref{assumption_enumi3}) are fulfilled. In addition, $C> 0$ denotes the operator norm of the injection $W \hookrightarrow V$; hereby, we exclude the trivial case $C=0$, since $C=0$ implies $W=\{0\}$ and by density $V=\{0\}$.\\
In Assumption \ref{assumption}.\ref{assumption_enumi3}) we did not assume $w_{1}\neq w_{2}$. If $w_{1}=w_{2}$, we set $\big(\kappa t+||w_{1}-w_{2}||_{W}^{-\frac{1}{\rho}}\big)^{-\rho}:=0$, which is reasonable, since: For any $x \in [0,\infty)$ the mapping $(0,\infty) \ni y \mapsto \left(x+y^{-\frac{1}{\rho}}\right)^{-\rho}$ can be extended continuously by zero in $y=0$.\\
Moreover, Assumption\ref{assumption}.\ref{assumption_enumi1}) yields that $W \in \B(V)$ and that if $Y \in \mathcal{M}(\Omega;V)$, with $\P(Y\in W)=1$ then the real-valued mapping $||Y||_{W}$ is up-to a $\P$-null-set well-defined and $\B(V)$-$\B(\R)$-measurable, see \cite[Remark 2.7]{ich3}.\\

The following estimates will play a fundamental role in this entire section, it is needed in the proofs of all of our main results, which are: Proposition \ref{prop_uniqueinvpropmeas}, Theorem \ref{theorem_slln}, Theorem \ref{theorem_clt} and Corollary \ref{corollary_clt}. The remaining results of this section simply serve to keep the exposition more structured, but are probably not of independent interest.\\
As mentioned introductory, proving the CLT requires the additional assumption $\rho>\frac{1}{2}$. It will be stated explicitly whenever this additional assumption is needed.

\begin{lemma}\label{lemma_fundamentalestimate} Let $x \in \mathcal{M}(\Omega;V)$ be an independent initial. Then the inequality
\begin{align}
\label{lemma_fundamentalestimateeq1}
||\X_{x}(t)||_{V} \leq C \kappa^{-\rho}(t-\alpha_{N(t)})^{-\rho},~\forall t>0,
\end{align}
takes place with probability one. In addition, if $y \in \mathcal{M}(\Omega;V)$ is another independent initial we have
\begin{align}
\label{lemma_fundamentalestimateeq2}
||\X_{x}(t)-\X_{y}(t)||_{V} \leq C \kappa^{-\rho}t^{-\rho},~\forall t>0,
\end{align}
almost surely.
\end{lemma}
\begin{proof} Let us start by proving (\ref{lemma_fundamentalestimateeq1}). To this end, let $(\tilde{\eta}_{m})_{m \in \mathbb{N}}\subseteq \mathcal{M}(\Omega;V)$ and $\tilde{x}\in \mathcal{M}(\Omega;V)$, assume $\tilde{\eta}_{m},\tilde{x}\in W$ almost surely and introduce $(\tilde{X}(t))_{t\geq 0}$ and $(\tilde{x}_{m})_{m \in \mathbb{N}_{0}}$ as the process and the sequence generated by  $((\beta_{m})_{m \in \mathbb{N}},(\tilde{\eta}_{m})_{m \in \mathbb{N}},\tilde{x},T)$ in $V$, respectively.\\
Then, note that $\tilde{x}_{m} \in W$ for all $m \in \mathbb{N}_{0}$ almost surely, since: If $m=0$ this is trivial and if it holds for an $m \in \mathbb{N}_{0}$, we have $\tilde{x}_{m+1}=T(\beta_{m+1})\tilde{x}_{m}+\tilde{\eta}_{m+1}$ and both summands are elements of $W$, for the first this follows from Assumption \ref{assumption}.\ref{assumption_enumi2}) and the induction hypothesis, and for the second this holds by construction.\\
Consequently, appealing to Assumption \ref{assumption}.\ref{assumption_enumi2}) again yields $\tilde{X}(t) \in W$ for all $t\geq 0$, almost surely, since clearly $\tilde{X}(t)=T(t-\alpha_{N(t)})\tilde{x}_{N(t)}$ for all $t \geq 0$ with probability one.\\
Hence, employing Assumption \ref{assumption}.\ref{assumption_enumi3}) and \ref{assumption_enumi4}) yields
\begin{align}
\label{lemma_fundamentalestimateproof1}
||\tilde{X}(t)||_{V}\leq C||T(t-\alpha_{N(t)})\tilde{x}_{N(t)}||_{W} \leq C \left(\kappa (t-\alpha_{N(t)})+||\tilde{x}_{N(t)}||_{W}^{-\frac{1}{\rho}}\right)^{-\rho} \leq C\kappa^{-\rho}(t-\alpha_{N(t)})^{-\rho},
\end{align}
for all $t>0$ almost surely.\\
Now let us infer (\ref{lemma_fundamentalestimateeq1}) from (\ref{lemma_fundamentalestimateproof1}). To this end, let $(\Gamma_{n})_{n \in \mathbb{N}}$, where $\Gamma_{n}:V\rightarrow V$, be a sequence of $\B(V)$-$\B(V)$-measurable mappings, such that
\begin{align}
\label{lemma_fundamentalestimateproof2}
\Gamma_{n}(V)\subseteq W,~\forall n \in \mathbb{N} \text{ and } \lim \limits_{n \rightarrow \infty} \Gamma_{n}(v)=v,~\forall v \in V.
\end{align}
Since $W$ is dense in $(V,||\cdot||_{V})$, such a sequence exists, see \cite[Lemma 3.12]{ich3}. Now, for every $n \in \mathbb{N}$, let $(X^{n}(t))_{t \geq 0}$ be the process generated by $((\beta_{m})_{m \in \mathbb{N}},\left(\Gamma_{n}(\eta_{m})\right)_{m \in \mathbb{N}},\Gamma_{n}(x),T)$ in $V$. Then, as $\Gamma_{n}(V)\subseteq W$, (\ref{lemma_fundamentalestimateproof1}) yields $||X^{n}(t)||_{V}\leq C\kappa^{-\rho}(t-\alpha_{N(t)})^{-\rho}$ for all $t>0$ and $n \in \mathbb{N}$ almost surely. (If one sets $(\tilde{\eta}_{m})_{m \in \mathbb{N}}=\left(\Gamma_{n}(\eta_{m})\right)_{m \in \mathbb{N}}$ and $\tilde{x}=\Gamma_{n}(x)$ for a given $n \in \mathbb{N}$, then $\tilde{X}=X^{n}$). Moreover, appealing to Lemma \ref{lemma_xcont}, while having in mind (\ref{lemma_fundamentalestimateproof2}), yields 
\begin{align*}
||\X_{x}(t)||_{V} \leq \lim \limits_{n \rightarrow \infty}  ||x-\Gamma_{n}(x)||_{V}+\sum \limits_{m=1}\limits^{N(t)}||\eta_{m}-\Gamma_{n}(\eta_{m})||_{V}+C\kappa^{-\rho}(t-\alpha_{N(t)})^{-\rho} =C\kappa^{-\rho}(t-\alpha_{N(t)})^{-\rho},
\end{align*}
for all $t>0$, with probability one. Consequently, (\ref{lemma_fundamentalestimateeq1}) is proven and it remains to verify (\ref{lemma_fundamentalestimateeq2}).\\
In addition, to the existing notations, let $\tilde{y}\in \mathcal{M}(\Omega;V)$, assume $\tilde{y}\in W$ almost surely and introduce $(\tilde{Y}(t))_{t\geq 0}$ and $(\tilde{y}_{m})_{m \in \mathbb{N}_{0}}$ as the process and the sequence generated by  $((\beta_{m})_{m \in \mathbb{N}},(\tilde{\eta}_{m})_{m \in \mathbb{N}},\tilde{y},T)$ in $V$, respectively.\\
Of course, we then also have $\tilde{Y}(t),\tilde{y}_{m}\in W$ for all $t \geq 0$ and $m \in \mathbb{N}_{0}$, with probability one. Now let us verify inductively that
\begin{align}
\label{lemma_fundamentalestimateeq3}
||\tilde{x}_{m}-\tilde{y}_{m}||_{W}\leq \left(\kappa \alpha_{m}+||\tilde{x}-\tilde{y}||_{W}^{-\frac{1}{\rho}}\right)^{-\rho},~\text{a.s. }\forall m \in \mathbb{N}.
\end{align}
If $m=0$, (\ref{lemma_fundamentalestimateeq3}) is even an equality. And if it holds for an $m \in \mathbb{N}_{0}$ we get by applying Assumption \ref{assumption}.\ref{assumption_enumi3}) and then the induction hypothesis that
\begin{align*}
||\tilde{x}_{m+1}-\tilde{y}_{m+1}||_{W} \leq \left(\kappa \beta_{m+1}+||\tilde{x}_{m}-\tilde{y}_{m}||_{W}^{-\frac{1}{\rho}}\right)^{-\rho} \leq \left(\kappa \alpha_{m+1}+||\tilde{x}-\tilde{y}||_{W}^{-\frac{1}{\rho}}\right)^{-\rho},
\end{align*}
with probability one, which proves (\ref{lemma_fundamentalestimateeq3}). Using this, while employing the services of \ref{assumption}.\ref{assumption_enumi3}) once more gives
\begin{eqnarray*}
||\tilde{X}(t)-\tilde{Y}(t)||_{V}
&\leq& ~ C||T(t-\alpha_{N(t)})\tilde{x}_{N(t)}-T(t-\alpha_{N(t)})\tilde{y}_{N(t)}||_{W} \\
&\leq& ~ C\left(\kappa (t-\alpha_{N(t)})+||\tilde{x}_{N(t)}-\tilde{y}_{N(t)}||_{W}^{-\frac{1}{\rho}}\right)^{-\rho} \\
&\leq& ~ C\left(\kappa (t-\alpha_{N(t)})+\kappa \alpha_{N(t)}+||\tilde{x}-\tilde{y}||_{W}^{-\frac{1}{\rho}}\right)^{-\rho} \\
&\leq& ~ C\left(\kappa t\right)^{-\rho},
\end{eqnarray*}
for all $t>0$ with probability one. Now, for every $n \in \mathbb{N}$, let $(Y^{n}(t))_{t \geq 0}$ be the process generated by $((\beta_{m})_{m \in \mathbb{N}},\left(\Gamma_{n}(\eta_{m})\right)_{m \in \mathbb{N}},\Gamma_{n}(y),T)$ in $V$. Then, as $\Gamma_{n}(V)\subseteq W$, the preceding calculation yields $||X^{n}(t)-Y^{n}(t)||_{V}\leq  C\left(\kappa t\right)^{-\rho}$ for all $t>0$ and $n \in \mathbb{N}$ almost surely. Finally, Lemma \ref{lemma_xcont} enables us to conclude that
\begin{align*}
||\X_{x}(t)-\X_{y}(t)||_{V}\leq C\left(\kappa t\right)^{-\rho}+||x-\Gamma_{n}(x)||_{V}+||y-\Gamma_{n}(y)||_{V}+2\sum \limits_{m=1}\limits^{N(t)}||\eta_{m}-\Gamma_{n}(\eta_{m})||_{V},
\end{align*}
for all $t>0$ and $n \in \mathbb{N}$ with probability one, which yields the claim by recalling (\ref{lemma_fundamentalestimateproof2}) and letting $n$ to infinity. 
\end{proof}

\begin{proposition}\label{prop_uniqueinvpropmeas} The transition function $P$ possesses a unique invariant probability measure, i.e. there is one, and only one, probability measure $\mu: \B(V)\rightarrow [0,1]$, such that
\begin{align}
\label{prop_uniqueinvpropmeaseq}
\int \limits_{V} P(t,v,B) \mu (dv)=\mu(B),~\forall t \geq 0,~B \in \B(V).
\end{align}
\end{proposition}
\begin{proof} Appealing to Theorem \ref{theorem_mp} as well as Lemma \ref{lemma_basicprop}.\ref{lemma_basicpropmpenumi1})-\ref{lemma_basicpropmpenumi3}) yields, by virtue of \cite[Theorem 1]{Szarek}, the existence of a unique invariant probability measure, if we can prove that
\begin{align*}
\liminf \limits_{t \rightarrow \infty} \frac{1}{t} \int \limits_{0}\limits^{t} P(||\X_{v}(t)||_{V}< \varepsilon)d\tau >0,~\forall \varepsilon>0,~v \in V.
\end{align*}
So fix $\varepsilon>0$ as well as $v \in V$ and recall the well-known fact that $\P(\tau-\alpha_{N(\tau)}>q)=\exp(-\theta q)\id_{[0,\tau)}(q)$ for all $\tau,q\in [0,\infty)$. Now, introduce $q:=\kappa^{-1}\varepsilon^{-\frac{1}{\rho}}C^{\frac{1}{\rho}}$.\\
Then we get by Lemma \ref{lemma_fundamentalestimate} that
\begin{eqnarray*}
	\liminf \limits_{t \rightarrow \infty} \frac{1}{t} \int \limits_{0}\limits^{t} P(||\X_{v}(\tau)||_{V}< \varepsilon)d\tau
	& \geq & ~ \liminf \limits_{t \rightarrow \infty} \frac{1}{t} \int \limits_{0}\limits^{t} P(C \kappa^{-\rho}(\tau-\alpha_{N(\tau)})^{-\rho}< \varepsilon)d\tau \\
	& = & ~ \liminf \limits_{t \rightarrow \infty} \frac{1}{t} \int \limits_{0}\limits^{t} P(\tau-\alpha_{N(\tau)}> q)d\tau\\
    & = & ~ \exp(-\theta q),
\end{eqnarray*}
which is obviously strictly positive.
\end{proof}

\begin{remark} In the remainder of this section, $\bar{\mu}:\B(V)\rightarrow [0,1]$, denotes the uniquely determined probability measure fulfilling (\ref{prop_uniqueinvpropmeaseq}). Moreover, we call an $\bar{x} \in \mathcal{M}(\Omega;V)$ which is an independent initial with $\P(\bar{x} \in B)=\bar{\mu}(B)$, for all $B \in \B(V)$, an independent, stationary initial.\\
As $\bar{\mu}$ is unique, it is ergodic, see \cite[Theorem 3.2.6]{Prato} for a proof and \cite[Theorem 3.2.4]{Prato} for a couple of useful equivalent definitions of ergodicity, commonly used in the literature.\\ Furthermore, if $\bar{x} \in \mathcal{M}(\Omega;V)$ is an independent, stationary initial, then the Markov process $(\X_{\bar{x}}(t))_{t \geq 0}$ is strictly stationary, see \cite[Lemma 8.11]{Kallenberg}. Moreover, $(\X_{\bar{x}}(t))_{t \geq 0}$ is also ergodic (in the sense that the shift invariant $\sigma$-algebra is $\P$-trivial), which one easily deduces from \cite[Prop. 2.2]{ergodicequi} by appealing to \cite[Theorem 3.2.4.ii)]{Prato}.\\
Finally, $L^{2}(\bar{\mu}):=L^{2}(V,\B(V),\bar{\mu})$ and for any $\psi \in L^{2}(\bar{\mu})$ we set $\overline{(\psi)}:= \int \limits_{V}\psi(v)\bar{\mu}(dv)$ and introduce $L^{2}_{0}(\bar{\mu}):=\{\psi \in L^{2}(\bar{\mu}):\overline{(\psi)}=0\}$.
\end{remark}

\begin{lemma}\label{lemma_boundL2} Let $x \in \mathcal{M}(\Omega;V)$ be an independent initial. Then $||\X_{x}(t)||_{V} \in L^{2}(\Omega)$ for all $t \in (0,\infty)$. In particular, the following assertions hold.
\begin{enumerate}[i)]
	\item\label{lemma_boundL2_enumi1} $\psi(\X_{\bar{x}}(t)) \in L^{2}(\Omega)$, for all $t \in [0,\infty)$, $\psi \in Lip(V)$ and independent stationary initials $\bar{x}\in \mathcal{M}(\Omega;V)$.
	\item\label{lemma_boundL2_enumi2} $Lip(V) \subseteq L^{2}(\bar{\mu})$.
\end{enumerate} 
\end{lemma}
\begin{proof} Let $t>0$ and $x \in \mathcal{M}(\Omega;V)$ be an independent initial. Then we get by employing the services of Lemma \ref{lemma_xcont} and Lemma \ref{lemma_fundamentalestimate} that
\begin{align*}
||\X_{x}(t)||_{V} \leq ||\X_{x}(t)-\X_{0}(t)||_{V}+||\X_{0}(t)||_{V} \leq C \kappa^{-\rho}t^{-\rho}+\sum \limits_{m=1}\limits^{N(t)}||\eta_{m}||_{V}
\end{align*}
almost surely. Consequently, $||\X_{x}(t)||_{V} \in L^{2}(\Omega)$ holds, if $\sum \limits_{m=1}\limits^{N(t)}||\eta_{m}||_{V} \in L^{2}(\Omega)$. But the latter is true by the Blackwell-Girshick equation, which is applicable since $(||\eta_{k}||)_{k \in \mathbb{N}}\subseteq L^{2}(\Omega)$ is i.i.d. and independent of $(N(t))_{t \geq 0}$, which is (as it is a Poisson process) in particular square integrable.\\
Now, note that, due to stationary, \ref{lemma_boundL2}.\ref{lemma_boundL2_enumi1}) holds for one $t \in [0,\infty)$ if and only if, it holds for every $t \in [0,\infty)$. So assume $t>0$, then we get $|\psi(\X_{\bar{x}}(t))|\leq L_{\psi}||\X_{\bar{x}}(t)||_{V}+|\psi(0)|$, which is already known to be square integrable. Finally, \ref{lemma_boundL2}.\ref{lemma_boundL2_enumi2}) follows from \ref{lemma_boundL2}.\ref{lemma_boundL2_enumi1}), since $||\psi||_{L^{2}(\bar{\mu})}^{2}=\mathbb{E}\left(\psi(\bar{x})^{2}\right)$.
\end{proof}

\begin{theorem}\label{theorem_slln} Let $\psi \in Lip(V)$ and $x \in \mathcal{M}(\Omega;V)$ be an independent initial. Then the convergence
\begin{align}
\label{theorem_sllneq}
\lim \limits_{t \rightarrow \infty } \frac{1}{t} \int \limits_{0} \limits^{t} \psi (\X_{x}(\tau)) d\tau =\overline{(\psi)},
\end{align}
takes place with probability one.
\end{theorem}
\begin{proof} Firstly, note that the left hand side integral exists, since Lemma \ref{lemma_meas} and Lemma \ref{lemma_basicprop} yield that $[0,t]\ni \tau \mapsto \psi(\X_{x}(\tau,\omega))$ is $\B([0,t])$-$\B(\R)$-measurable and for $\P$-a.e. $\omega \in \Omega$ bounded, respectively.\\
Now let $\bar{x} \in \mathcal{M}(\Omega;V)$ be an independent stationary initial. Then appealing to \cite[Theorem 3.3.1]{Prato} yields
\begin{align}
\label{theorem_sllnproofeq1}
\lim \limits_{t \rightarrow \infty } \frac{1}{t} \int \limits_{0} \limits^{t} \psi (\X_{\bar{x}}(\tau)) d\tau = \overline{(\psi)},
\end{align}
almost surely, for all $\psi \in Lip(V)$. (This theorem is indeed applicable, since $\bar{\mu}$ is ergodic, $(\X_{\bar{x}}(t))_{t \geq 0}$ is stationary, stochastically continuous and since $Lip(V)\subseteq L^{2}(\bar{\mu})$.)\\
Conclusively, recalling Lemma \ref{lemma_fundamentalestimate} gives
\begin{align*}
\lim \limits_{t \rightarrow \infty } \left|\frac{1}{t} \int \limits_{0} \limits^{t} \psi (\X_{\bar{x}}(\tau))- \psi (\X_{x}(\tau)) d\tau\right| \leq L_{\psi}C\kappa^{-\rho}\lim \limits_{t \rightarrow \infty } \frac{1}{t} \int \limits_{1} \limits^{t} \tau^{-\rho} d\tau=0,
\end{align*}
almost surely, which yields combined with (\ref{theorem_sllnproofeq1}) the claim.
\end{proof}

The task ahead of us that remains is proving the CLT, which will be achieved by the results in \cite{Holzmann}. Applying the results in \cite{Holzmann} requires to extend the family of mappings $(Q(t))_{t \geq 0}$ to a linear, time-continuous, contractive semigroup on $L^{2}(\bar{\mu})$. To aid the reader who is not too familiar with Markov processes, let us outline why this is possible.

\begin{remark}\label{remark_extension} Let $t \in [0,\infty)$ be given. Then for any $\hat{V} \in \B(V)$, with $\bar{\mu}(\hat{V})=1$, we get by the invariance of $\bar{\mu}$ that there is a set $\tilde{V} \in \B(V)$, with $\mu(\tilde{V})=1$, such that $\P(\X_{v}(t)\in \hat{V})=1,~\forall v \in \tilde{V}$.\\
Moreover, if $\psi = \id_{B}$, where $B \in \B(V)$, then the invariance of $\bar{\mu}$ gives
\begin{align}
\label{remark_extensioneq1}
\int \limits_{V} \mathbb{E} \psi(\X_{v}(t))\bar{\mu}(dv)= \int \limits_{V} \psi(v)\bar{\mu}(dv).
\end{align}
Moreover, by linearity in $\psi$, (\ref{remark_extensioneq1}) also holds for all step functions. Now let $\psi \in L^{2}(\bar{\mu})$ be arbitrary, then there are step functions $(\psi_{m})_{m \in \mathbb{N}}$ with $\lim \limits_{m \rightarrow \infty}\psi_{m}=\psi$ in $L^{2}(\bar{\mu})$ and $\bar{\mu}$-a.e. Hence, for $\mu$-a.e. $v \in V$ we get $\lim \limits_{m \rightarrow \infty} \psi_{m}(\X_{v}(t))= \psi(\X_{v}(t))$ a.s. Consequently, applying Fatou's Lemma (twice) and (\ref{remark_extensioneq1}) yields
\begin{align*}
\int \limits_{V} \mathbb{E}\left( \psi(\X_{v}(t))^{2}\right)\bar{\mu}(dv) \leq \liminf \limits_{m \rightarrow \infty} \int \limits_{V} \mathbb{E}\left( \psi_{m}(\X_{v}(t))^{2}\right)\bar{\mu}(dv) = \liminf \limits_{m \rightarrow \infty} \int \limits_{V} \psi_{m}(v)^{2}\bar{\mu}(dv) =||\psi||_{L^{2}(\bar{\mu})}^{2}<\infty.
\end{align*} 
Hence, for $\bar{\mu}$-almost every $v \in V$, $\mathbb{E}\psi(\X_{v}(t))$ exists and we infer from Jensen's inequality that
\begin{align}
\label{remark_extensioneq2}
\int \limits_{V} \big(\mathbb{E} \psi(\X_{v}(t))\big)^{2}\bar{\mu}(dv) \leq ||\psi||_{L^{2}(\bar{\mu})}^{2},~\forall \psi \in L^{2}(\bar{\mu}).
\end{align}
Consequently, we can extend the domain of each $Q(t)$ to $L^{2}(\bar{\mu})$, i.e. from now on $Q(t):L^{2}(\bar{\mu})\rightarrow L^{2}(\bar{\mu})$, with $(Q(t)\psi)(v):= \mathbb{E} \psi(\X_{v}(t))$, for all $t \in [0,\infty)$, $v \in V$ and $\psi \in L^{2}(\bar{\mu})$.\\
Using this and Theorem \ref{theorem_mp}.\ref{theorem_mpenumi4}) yields that $(Q(t))_{t \geq 0}$ is a linear, contractive semigroup on $L^{2}(\bar{\mu})$, see \cite[Theorem 1, p. 381]{Yosida} for a detailed proof.
\end{remark}

It seems to be mathematical common knowledge that this semigroup is (due to stochastic continuity and contractivity) time-continuous. But, the present author was unable to find any proof of this assertion, therefore let's do that:

\begin{lemma}\label{lemma_extension} The family of mappings $(Q(t))_{t \geq 0}$ is a linear, time-continuous contractive semigroup on $L^{2}(\bar{\mu})$.
\end{lemma}
\begin{proof} In light of Remark \ref{remark_extension}, it remains to prove the time continuity. So let $(h_{m})_{m \in \mathbb{N}}$ be a null-sequence, let $t \in [0,\infty)$ and assume w.l.o.g. that $t+h_{m}\geq 0$ for all $m \in \mathbb{N}$. Now let $\psi \in L^{2}(\bar{\mu})$, choose $\varepsilon>0$ and $\varphi \in C_{b}(V)$ such that $||\psi-\varphi||_{L^{2}(\bar{\mu})}<\frac{\varepsilon}{2}$. Then, by stochastic continuity of $(\X_{v}(t))_{t\geq 0}$, and passing to a subsequence if necessary, we have $\lim \limits_{m \rightarrow \infty}\varphi(\X_{v}(t+h_{m}))=\varphi(\X_{v}(t))$ almost surely. Thus, the boundedness of $\varphi$ yields (by dominated convergence) that  $\lim \limits_{m \rightarrow \infty}(Q(t+h_{m})\varphi)(v)=(Q(t)\varphi)(v)$ for all $v \in V$. Consequently, employing Lebesgue's theorem once more gives  $\lim \limits_{m \rightarrow \infty}Q(t+h_{m})\varphi=Q(t)\varphi$ in $L^{2}(\bar{\mu})$. Using this, as well as the contractivity of $Q$ gives
\begin{align*}
\lim \limits_{m \rightarrow \infty} ||Q(t+h_{m})\psi-Q(t)\psi||_{L^{2}(\bar{\mu})}\leq 2||\psi-\varphi||_{L^{2}(\bar{\mu})}+\lim \limits_{m \rightarrow \infty} ||Q(t+h_{m})\varphi-Q(t)\varphi||_{L^{2}(\bar{\mu})}\leq \varepsilon,
\end{align*}
which yields the desired time continuity, as $\varepsilon>0$ was arbitrary.
\end{proof}

\begin{lemma}\label{lemma_cltprepbpund} Let $\psi \in Lip(V)$ and set $\psi_{c}:=\psi-\overline{(\psi)}$. Then $\psi_{c} \in L^{2}_{0}(\bar{\mu})$ and
\begin{align*}
||Q(t)\psi_{c}||_{L^{2}(\bar{\mu})}\leq L_{\psi}C \kappa^{-\rho}t^{-\rho},
\end{align*}
for all $t>0$.
\end{lemma}
\begin{proof} Clearly, $\psi_{c} \in Lip(V)$, thus $\psi_{c} \in L^{2}(\bar{\mu})$ by Lemma \ref{lemma_boundL2}.\ref{lemma_boundL2_enumi2}). Moreover, $\psi_{c}$ is obviously centered. In addition, by stationary we get $\overline{(\psi)}= \mathbb{E}\psi(\X_{\bar{x}}(t))$, where $\bar{x}\in \mathcal{M}(\Omega;V)$ is an independent, stationary initial. Using this and invoking Lemma \ref{lemma_fundamentalestimate} yields
\begin{align*}
||Q(t)\psi_{c}||_{L^{2}(\bar{\mu})}^{2} = \int \limits_{V} \big(\mathbb{E}[\psi(\X_{v}(t))-\psi(\X_{\bar{x}}(t))]\big)^{2}\bar{\mu}(dv) \leq (L_{\psi}C\kappa^{-\rho}t^{-\rho})^{2}
\end{align*}
and the claim follows.
\end{proof}

\begin{theorem}\label{theorem_clt} Assume $\rho > \frac{1}{2}$, let $\psi \in Lip(V)$ and $x \in \mathcal{M}(\Omega;V)$ be an independent initial. Then there is a $\sigma^{2}(\psi) \in [0,\infty)$ such that
\begin{align}
\label{theorem_clteq1}
\lim \limits_{t \rightarrow \infty }\frac{1}{\sqrt{t}}\left( \int \limits_{0}\limits^{t}\psi(\X_{x}(\tau))d\tau-t\overline{(\psi)}\right)=Y\sim N(0,\sigma^{2}(\psi)) ,
\end{align}
in distribution. Moreover, we have
\begin{align}
\label{theorem_clteq2}
\sigma^{2}(\psi) :=  \lim \limits_{t \rightarrow \infty}\frac{1}{t}\mathbb{E}\left( \int \limits_{0}\limits^{t}\psi_{c}(\X_{\bar{x}}(\tau))d\tau\right)^{2}=\lim \limits_{t \rightarrow \infty}\frac{1}{t}\Var\left( \int \limits_{0}\limits^{t}\psi(\X_{\bar{x}}(\tau))d\tau\right),
\end{align}
where $\bar{x}\in \mathcal{M}(\Omega;V)$ is an arbitrary stationary, independent initial and $\psi_{c}:=\psi-\overline{(\psi)}$.
\end{theorem}
\begin{proof} Appealing to Lemma \ref{lemma_cltprepbpund} gives $\psi_{c} \in L^{2}_{0}(\bar{\mu})$ as well as
\begin{align*}
\int \limits_{1}\limits^{\infty} \frac{1}{\sqrt{t}}||Q(t)\psi_{c}||_{L^{2}(\bar{\mu})}dt \leq L_{\psi}C \kappa^{-\rho} \int \limits_{1}\limits^{\infty} t^{-\rho-\frac{1}{2}}dt,
\end{align*}
which is finite, since $\rho>\frac{1}{2}$. Consequently, as we already know that $(\X_{\bar{x}}(t))_{t \geq 0}$ is a stationary, ergodic, $(\F_{t}^{\bar{x}})_{t \geq 0}$-progressive Markov process with time-continuous, contractive semigroup $(Q(t))_{t \geq 0}$, we get by \cite[Corollary 3.2 and Theorem 3.1]{Holzmann} that
\begin{align}
\label{theorem_cltproofeq1}
\lim \limits_{t \rightarrow \infty }\frac{1}{\sqrt{t}} \int \limits_{0}\limits^{t}\psi_{c}(\X_{\bar{x}}(\tau))d\tau=Y\sim N(0,\sigma^{2}(\psi)),
\end{align}
in distribution and that $\sigma^{2}(\psi)$ is indeed given by the first equality in (\ref{theorem_clteq2}). Moreover, the second equality in (\ref{theorem_clteq2}) is trivial, since $\psi_{c}(\X_{\bar{x}}(\tau))=\psi(\X_{\bar{x}}(\tau))-\mathbb{E}(\psi(\X_{\bar{x}}(\tau)))$ by stationarity.\\
Now, note that clearly
\begin{align*}
\frac{1}{\sqrt{t}}\left( \int \limits_{0}\limits^{t}\psi(\X_{x}(\tau))d\tau-t\overline{(\psi)}\right) = \frac{1}{\sqrt{t}}\int \limits_{0}\limits^{t}\psi(\X_{x}(\tau))-\psi(\X_{\bar{x}}(\tau))d\tau+\frac{1}{\sqrt{t}} \int \limits_{0}\limits^{t}\psi_{c}(\X_{\bar{x}}(\tau))d\tau,~\forall t>0
\end{align*}
which yields, in light of (\ref{theorem_cltproofeq1}), that (\ref{theorem_clteq1}) holds, if the first summand in the previous express converges almost surely to zero. But recalling that $\rho>\frac{1}{2}$ and invoking Lemma \ref{lemma_fundamentalestimate} yields
\begin{align*}
\lim \limits_{t \rightarrow \infty } \left|\frac{1}{\sqrt{t}} \int \limits_{0} \limits^{t} \psi (\X_{x}(\tau))- \psi (\X_{\bar{x}}(\tau)) d\tau\right| \leq L_{\psi}C\kappa^{-\rho}\lim \limits_{t \rightarrow \infty } \frac{1}{\sqrt{t}} \int \limits_{1} \limits^{t} \tau^{-\rho} d\tau=0,
\end{align*}
with probability one.
\end{proof}

Now this section concludes by summarizing Theorem \ref{theorem_slln} and Theorem \ref{theorem_clt} for the probably most prominent Lipschitz continuous map from $V$ to $\R$, namely $||\cdot||_{V}$.

\begin{corollary}\label{corollary_clt} Let $x \in \mathcal{M}(\Omega;V)$ be an independent initial and $\bar{x}\in \mathcal{M}(\Omega;V)$ a stationary independent initial. Then the following assertions hold.
\begin{enumerate}[i)]
	\item $\lim \limits_{t \rightarrow \infty } \frac{1}{t} \int \limits_{0} \limits^{t} ||\X_{x}(\tau)||_{V} d\tau =\nu$ with probability one, where $\nu:=\int \limits_{V} ||v||_{V} \bar{\mu}(dv)=\mathbb{E}||\bar{x}||_{V}$.
	\item\label{corollary_clt_enumi2} If $\rho>\frac{1}{2}$, then $\lim \limits_{t \rightarrow \infty }\frac{1}{\sqrt{t}}\left( \int \limits_{0}\limits^{t}||\X_{x}(\tau)||_{V}d\tau-t\nu\right)=Y\sim N(0,\sigma^{2})$ in distribution, where $\sigma^{2}\in [0,\infty)$, with $\sigma^{2} =\lim \limits_{t \rightarrow \infty}\frac{1}{t}\Var\left( \int \limits_{0}\limits^{t}||\X_{\bar{x}}(\tau)||_{V}d\tau\right)$.
\end{enumerate}
\end{corollary}

\section{Examples and a useful Criteria}
\label{sec_examples}

The first result of this section is the introductory mentioned differential-inequality-result, which is probably not only in our examples useful to verify Assumption \ref{assumption}.\ref{assumption_enumi3}). Even though this result seems to be in common use, we were unable to find it anywhere in the literature, stated precisely as we need it and with a rigorous proof. Therefore, the simple proof will be given. Once this is achieved we proceed with our ODE example and devote the remainder of this section to the $p$-Laplacian example.

\begin{lemma}\label{lemma_diffinequality} Let $f:[0,\infty)\rightarrow [0,\infty)$ be locally Lipschitz continuous on $[0,\infty)$. Moreover, assume that there are constants $\kappa,\tilde{\rho} \in (0,\infty)$ such that
\begin{align}
\label{lemma_diffinequalityeq1}
f^{\prime}(t)\leq - \kappa \tilde{\rho} f(t)^{1+\frac{1}{\tilde{\rho}}},
\end{align}
for a.e. $t \in (0,\infty)$. Then we have
\begin{align}
\label{lemma_diffinequalityeq2}
f(t)\leq \left(\kappa t+f(0)^{-\frac{1}{\tilde{\rho}}}\right)^{-\tilde{\rho}},
\end{align}
for all $t \in [0,\infty)$.
\end{lemma}
\begin{proof} Firstly, as $f$ is a real-valued, locally Lipschitz continuous function it is indeed differentiable almost everywhere. Moreover, (\ref{lemma_diffinequalityeq1}) yields that $f$ is monotonically decreasing.\\ 
Now set $I:=\inf\{t\geq0:~ f(t)=0\}$. If $I=0$, then $f(t)=0$ for all $t>0$ and by continuity for all $t \geq 0$. Consequently, in this case (\ref{lemma_diffinequalityeq2}) trivially holds. So assume $I>0$ and let $\tilde{I} \in [0,I)$ be arbitrary but fixed and introduce $F:[0,\tilde{I}]\rightarrow [0,\infty)$ with $F(t):=f(t)^{-\frac{1}{\tilde{\rho}}}$. As $f(t)\geq f(\tilde{I})>0$ for all $t \in [0,\tilde{I}]$, $F$ is, as is the composition of Lipschitz continuous functions, itself Lipschitz continuous. Consequently, we get
\begin{align*}
F(t)-F(0)= \int \limits_{0}\limits^{t} F^{\prime}(\tau)d\tau =-\frac{1}{\tilde{\rho}} \int \limits_{0}\limits^{t} f(\tau)^{-\frac{1}{\tilde{\rho}}-1}f^{\prime}(\tau)d\tau\geq \kappa t,~\forall t \in [0,\tilde{I}].
\end{align*}
Thus (\ref{lemma_diffinequalityeq2}) holds on $[0,\tilde{I}]$ and as $\tilde{I}$ was arbitrary, it holds on $t \in [0,I)$. Finally, if $I=\infty$ the proof is complete and if $I<\infty$, the infimum is (by continuity) a minimum and by monotonicity $f=0$ on $[I,\infty)$, in which case (\ref{lemma_diffinequalityeq2}) is trivial.
\end{proof}

\begin{example}\label{remarkex1} Let $\rho \in (0,\infty)$ and introduce the family of mappings $(T(t))_{t \geq 0}$, where $T(t):\R\rightarrow\R$, by $T(t)v:=sgn(v)\left(t+|v|^{-\frac{1}{\rho}}\right)^{-\rho}$ for all $v \in \mathbb{R}$. Then, obviously $T(0)v=sgn(v)|v|=v$ and a direct calculation shows that $T(\cdot)v$ fulfills the ODE 
\begin{align*}
y^{\prime}(t)=-\rho y(t)|y(t)|^{\frac{1}{\rho}},~\forall t \in (0,\infty) \text{ and } y(0)=v.
\end{align*}
Moreover, we have	
\begin{align*}
T(t)(T(h)v)=sgn(T(h)v)\left(t+|T(h)v|^{-\frac{1}{\rho}}\right)^{-\rho}=sgn(v)\left(t+h+|v|^{-\frac{1}{\rho}}\right)^{-\rho}=T(t+h)v,
\end{align*}
for all $t,h \in [0,\infty)$ and $v \in \R$. Thus, as $t \mapsto T(t)v$ is trivially continuous, $(T(t))_{t \geq 0}$ is a time-continuous semigroup.\\
Now set $\kappa:=2^{-\frac{1}{\rho}}$ and let us verify Assumption \ref{assumption}.\ref{assumption_enumi3}), with $V=W=\R$. Doing this requires to prove
\begin{enumerate}[i)]
	\item\label{remarkex1_enumi1} $T(t)u_{1}+T(t)u_{2}\leq \left(\kappa t+(u_{1}+u_{2})^{-\frac{1}{\rho}}\right)^{-\rho}$, for all $t \in [0,\infty)$, $u_{1},u_{2} \geq 0$ and
	\item\label{remarkex1_enumi2} $T(t)u_{1}-T(t)u_{2}\leq \left(\kappa t+(u_{1}-u_{2})^{-\frac{1}{\rho}}\right)^{-\rho}$, for all $t \in [0,\infty)$, $u_{1},u_{2} \geq 0$ with $u_{1}\geq u_{2}$.
\end{enumerate}
Proof of \ref{remarkex1_enumi1}). Firstly, the convexity of $[0,\infty) \ni x \mapsto x^{1+\frac{1}{\rho}}$ yields $x^{1+\frac{1}{\rho}}+y^{1+\frac{1}{\rho}}\geq 2^{-\frac{1}{\rho}}(x+y)^{1+\frac{1}{\rho}}$ for all $x,y \in [0,\infty)$. Now set $f(t):=T(t)u_{1}+T(t)u_{2}$, for all $t \in [0,\infty)$. Then we get
\begin{align*}
f^{\prime}(t)= -\rho \left((T(t)u_{1})^{1+\frac{1}{\rho}}+(T(t)u_{2})^{1+\frac{1}{\rho}}\right)\leq -\rho \kappa\left(T(t)u_{1}+T(t)u_{2}\right)^{1+\frac{1}{\rho}}=-\rho \kappa f(t)^{1+\frac{1}{\rho}},~\forall t>0.
\end{align*}
Consequently, as $f$ is (particularly locally) Lipschitz continuous, \ref{remarkex1_enumi1}) follows from Lemma \ref{lemma_diffinequality}.\\
Proof of \ref{remarkex1_enumi2}). Firstly, it is easily verified that $x^{1+\frac{1}{\rho}}-y^{1+\frac{1}{\rho}}\geq (x-y)^{1+\frac{1}{\rho}}\geq \kappa (x-y)^{1+\frac{1}{\rho}}$ for all $x \geq y \geq 0$. Moreover, note that $T(t)u_{1}\geq T(t)u_{2}$, since $u_{1}\geq u_{2}\geq 0$. Now, set $f(t):=T(t)u_{1}-T(t)u_{2}$, then we get
\begin{align*}
f^{\prime}(t)=-\rho \left((T(t)u_{1})^{1+\frac{1}{\rho}}-(T(t)u_{2})^{1+\frac{1}{\rho}}\right)\leq -\rho \kappa \left(T(t)u_{1}-T(t)u_{2}\right)^{1+\frac{1}{\rho}}=-\rho \kappa f(t)^{1+\frac{1}{\rho}},~\forall t>0.
\end{align*}
Consequently, employing Lemma \ref{lemma_diffinequality} once more yields \ref{remarkex1_enumi2}).\\
Now, one easily infers from \ref{remarkex1_enumi1}), \ref{remarkex1_enumi2}) and $T(t)(-v)=-T(t)v$, for all $v \in \R$ that
\begin{align*}
|T(t)u-T(t)v|\leq \left(\kappa t+|u-v|^{-\frac{1}{\rho}}\right)^{-\rho},~\forall t \in [0,\infty),~u,v \in \R.
\end{align*}
In particular, $(T(t))_{t \geq 0}$ is contractive and by construction we have $T(t)0=0$. Consequently, $(T(t))_{t \geq 0}$ is a time-continuous, contractive semigroup fulfilling Assumption \ref{assumption} with $V=W=\R$. Now, let $(\eta_{m})_{m \in \mathbb{N}}\subseteq L^{2}(\Omega)$ be an i.i.d. sequence and let $(\beta_{m})_{m \in \mathbb{N}}$ be an i.i.d. sequence of $Exp(\theta)$-distributed random variables, where $\theta \in (0,\infty)$. In addition, assume that both sequences are independent of each other and let, for any independent initial $x \in \mathcal{M}(\Omega;\R)$, $\X_{x}:[0,\infty)\times \Omega \rightarrow \R$ denote the process generated by $((\beta_{m})_{m \in \mathbb{N}},(\eta_{m})_{m \in \mathbb{N}},x,T)$ in $\R$. Then, as the identity is Lipschitz continuous, it follows from Theorem \ref{theorem_slln} and Theorem \ref{theorem_clt} that
\begin{enumerate}[i)]\setcounter{enumi}{2}
	\item\label{remarkex1_enumi3} $\lim \limits_{t \rightarrow \infty } \frac{1}{t} \int \limits_{0} \limits^{t} \X_{x}(\tau) d\tau =\mathbb{E}\bar{x}$ a.s., for any independent initial $x \in \mathcal{M}(\Omega;\R)$ where $\bar{x}\in \mathcal{M}(\Omega;\R)$ is a stationary, independent initial, and
	\item\label{remarkex1_enumi4} if in addition $\rho>\frac{1}{2}$, then we have $\lim \limits_{t \rightarrow \infty }\frac{1}{\sqrt{t}}\left( \int \limits_{0}\limits^{t}\X_{x}(\tau)d\tau-t\mathbb{E}\bar{x}\right)=Y\sim N(0,\sigma^{2})$ in distribution,  for any independent initial $x \in \mathcal{M}(\Omega;\R)$, where $\sigma^{2} =  \lim \limits_{t \rightarrow \infty}\frac{1}{t}\Var\left( \int \limits_{0}\limits^{t}\X_{\bar{x}}(\tau)d\tau\right)$.
\end{enumerate} 
Now let us demonstrate that the assumption $\rho>\frac{1}{2}$ in \ref{remarkex1_enumi4}) cannot be dropped. To this end, assume $\eta_{k}=0$ for all $k \in \mathbb{N}$, then $\X_{x}(t)=T(t)x$ for any independent initial $x \in \mathcal{M}(\Omega;\R)$. Since $T(t)0=0$ for all $t \geq 0$, $\bar{x}=0$ is the (in this case even almost surely unique) stationary, independent initial. Consequently, we have $\mathbb{E}\bar{x}=\Var\left( \int \limits_{0}\limits^{t}\X_{\bar{x}}(\tau)d\tau\right)=0$ and \ref{remarkex1_enumi4}), with $x=1$ and without additional assuming $\rho>\frac{1}{2}$, would imply
\begin{align}
\label{remarkex1_eq1}
\lim \limits_{t \rightarrow \infty }\frac{1}{\sqrt{t}} \int \limits_{0}\limits^{t}\left(\tau+1\right)^{-\rho}d\tau=0,~\forall \rho>0,
\end{align}
which is now, due to the lack of randomness, simply convergence in $\R$. But obviously, (\ref{remarkex1_eq1}) is true if and only if $\rho>\frac{1}{2}$. 
\end{example}

Even though the semigroup considered in the previous example only acted on $\R$ and not an infinite dimensional Banach space, it is worth mentioning that neither \ref{remarkex1}.\ref{remarkex1_enumi3}) nor \ref{remarkex1}.\ref{remarkex1_enumi4}) are trivial.\\ 

Now let us turn to the weighted $p$-Laplacian example, in which case the semigroup acts on an infinite dimensional Banach space.\\
Throughout the remainder of this section, let $n \in \mathbb{N}\setminus \{1\}$  and $\emptyset \neq  S  \subseteq \mathbb{R}^{n}$ be a non-empty, open, connected and bounded sets of class $C^{1}$.
Moreover, let $p \in (2,\infty)$ and set $L^{q}(S,\mathbb{R}^{m}):=L^{q}(S,\B(S),\lambda;\mathbb{R}^{m})$, for any $q \in [1,\infty]$ and $m \in \mathbb{N}$, where $\lambda$ denotes the Lebesgue measure. This is further abbreviated by $L^{q}(S)$, if $m=1$. In addition, introduce $L^{q}_{0}(S):=\{f\in L^{q}(S):~\overline{(f)}=0\}$, where $\overline{(f)}:=\frac{1}{\lambda(S)}\int \limits_{S}fd\lambda$. Clearly, $(L^{q}_{0}(S),||\cdot||_{L^{q}(S)})$ is a separable Banach space.\\ 
Now, let $\gamma: S   \rightarrow (0,\infty)$ be such that $\gamma \in L^{\infty}( S )$, $\gamma^{\frac{2}{2-p}} \in L^{1}( S)$ and assume that there is an $A_{p}$-Muckenhoupt weight (see, \cite[page 4]{ich1}) $\gamma_{0}:\mathbb{R}^{n}\rightarrow \mathbb{R}$ such that $\gamma_{0}|_{ S }=\gamma$ a.e.  on $S$. Moreover, we introduce the weighted Sobolev space $W_{\gamma}^{1,p}( S )$ as
\begin{align*}
W_{\gamma }^{1,p}( S ):=\{f \in L^{p}( S ): f \text{ is weakly diff. and } ~\gamma^{\frac{1}{p}}\nabla f \in L^{p}( S;\mathbb{R}^{n})\}. 
\end{align*} 
In addition, whenever $q \in [1,\infty)$, $W^{1,q}(S)$ denotes the usual Sobolev space and $C_{S,q}$ is the Poincar\'{e} constant of $S$ in $L^{q}_{0}(S)$, i.e. the smallest constant such that $||\varphi||_{L^{q}(S)}\leq C_{S,q}||\nabla \varphi ||_{L^{q}(S)}$ for all \linebreak$\varphi \in W^{1,q}(S)\cap L^{q}_{0}(S)$.\\
Throughout the sequel, $|\cdot|_{n}$ is the Euclidean norm on $\mathbb{R}^{n}$ and for any $x,y\in\R^{n}$, $x\cdot y$ denotes the canonical inner product of these vectors.\\
In the sequel, we frequently apply results from \cite{mazon} and \cite{ich1}. Applying them requires the assumption $\gamma^{\frac{1}{1-p}}\in L^{1}(S)$, which is indeed easily inferred from $\gamma^{\frac{2}{2-p}}\in L^{1}(S)$.\\
The following weighted $p$-Laplace operator is the central object of the remainder of this paper:

\begin{definition}\label{definifition_plaplaceop} Let $A: D(A)\rightarrow 2^{L^{1}(S)}$ be defined by: $(f,\hat{f}) \in A$ if and only if the following assertions hold.
	\begin{enumerate}
		\item $f \in W^{1,p}_{\gamma}( S ) \cap L^{\infty}( S )$. 
		\item $\hat{f} \in L^{1}( S )$.
		\item $\int \limits_{ S}  \gamma|\nabla f|_{n}^{p-2}\nabla f\cdot\nabla \varphi  d \lambda = \int \limits_{ S } \hat{f} \varphi d \lambda$ for all $\varphi\in W^{1,p}_{\gamma }( S )\cap L^{\infty}( S )$.
	\end{enumerate}
Moreover, $\A:D(\A)\rightarrow 2^{L^{1}(S)}$ denotes the closure of $A$, i.e. $(f,\hat{f})\in \A$ if there is a sequence $((f_{m},\hat{f}_{m}))_{m \in \mathbb{N}}\subseteq A$ such that $\lim \limits_{m \rightarrow \infty} (f_{m},\hat{f}_{m})=(f,\hat{f})$, in $L^{1}(S)\times L^{1}(S)$
\end{definition}

One verifies that $A$ is single-valued, see \cite[Lemma 3.1]{ich1}. In addition, if one chooses $\gamma=1$ on $S$, then $A$ is simply the $p$-Laplacian operator with Neumann boundary conditions.\\
Moreover, it is possible to determine the closure explicitly, see \cite[Proposition 3.6]{mazon}. But the explicit description is fairly technical and not needed for our purposes, therefore it will be omitted. What is important to our purposes is that $\A$ is densely defined and m-accretive, see \cite[Section 3]{mazon}\footnote{This is also stated in \cite[Theorem 2.3]{ich1}, which summarizes the highlights of \cite[Section 3]{mazon}.}.
Consequently, recalling Remark \ref{remark_msee}, we can introduce $(T_{\A}(t))_{t \geq 0}$ as the semigroup associated to $\A$, which is, according to the same remark, a time-continuous, contractive semigroup on $L^{1}(S)$. In fact,  $(T_{\A}(t))_{t \geq 0}$ even forms a family of strong solutions, not just of mild ones, see \cite[Section 3]{mazon}.\\
The following three results enable us to apply the result of Sections \ref{sec_mp} and \ref{sec_sllnclt} to the current setting, which is achieved in Theorem \ref{theoremplaplace}. In particular, in Proposition \ref{prop_plaplacebound} it is demonstrated how to use Lemma \ref{lemma_diffinequality} in the current situation.

\begin{lemma}\label{lemma_invsssg} For each $q \in [1,\infty)$, the space $L^{q}_{0}(S)$ is invariant w.r.t. $T_{\A}(t)$. Moreover, the restriction of $T_{\A}$ to $L^{q}_{0}(S)$ is a time-continuous, contractive semigroup on $(L^{q}_{0}(S),||\cdot||_{L^{q}(S)})$ which fulfills $T_{\A}(t)0=0$ for all $t \in [0,\infty)$.
\end{lemma}
\begin{proof} Let $u\in L^{q}_{0}(S)$, then \cite[Lemma 3.4]{ich1} yields $\overline{(T_{\A}(t)u)}=0$ and by \cite[Lemma 3.3.2]{ich1} we get $T_{\A}(t)u \in L^{q}(S)$ for all $t \geq 0$. Consequently, $T_{\A}(t)u \in L^{q}_{0}(S)$ for all $t \geq 0$. In addition, $(T_{\A}(t))_{t \geq 0}$ inherits the semigroup property of $(T_{\A}(t))_{t\geq 0}$ and appealing to \cite[Lemma 3.3.1]{ich1} yields that $(T_{\A}(t))_{t \geq 0}$ is contractive. Moreover, $T_{\A}(t)0=0$ is inferred easily from $0 \in D(A)$, $A0=0$.\\
It remains to prove the time-continuity. So let $(h_{m})_{m \in \mathbb{N}}$ be a null-sequence, $t \geq 0$ and $\varepsilon>0$ be given, and assume w.l.o.g. that $t+h_{m}\geq 0$ for all $m \in \mathbb{N}$. Moreover, choose $v \in L^{\infty}_{0}(S)$ such that $||u-v||_{L^{q}(S)}<\frac{\varepsilon}{2}$. Then we get by the time continuity of $T_{\A}$, and by passing to a subsequence if necessary, that $\lim \limits_{m \rightarrow \infty} T_{\A}(t+h_{m})v=T_{\A}(t)v$ almost everywhere. In addition, invoking \cite[Lemma 3.3.3]{ich1} gives $||T_{q}(t+h_{m})v||_{L^{\infty}(S)}\leq ||v||_{L^{\infty}(S)}$ for all $m \in \mathbb{N}$ and employing dominated convergence gives \linebreak$\lim \limits_{m \rightarrow \infty} T_{\A}(t+h_{m})v=T_{\A}(t)v$ w.r.t. $||\cdot||_{L^{q}(S)}$. Conclusively, we get by contractivity that
\begin{align*}
\lim \limits_{m \rightarrow \infty} ||T_{\A}(t+h_{m})u-T_{\A}(t)u||_{L^{q}(S)}\leq 2||u-v||_{L^{q}(S)}+\lim \limits_{m \rightarrow \infty} ||T_{\A}(t+h_{m})v-T_{\A}(t)v||_{L^{q}(S)} \leq \varepsilon,
\end{align*}
which yields the desired time continuity.
\end{proof}

\begin{lemma}\label{lemma_lipdifae} Let $u,v \in L^{2}_{0}(S)\cap D(A)$ and $f:[0,\infty)\rightarrow [0,\infty)$, with $f(t):=||T_{\A}(t)u-T_{\A}(t)v||^{2}_{L^{2}(S)}$ for all $t \geq 0$. Then $f$ is locally Lipschitz continuous. Consequently, it is differentiable almost everywhere and we have
\begin{align}
\label{prop_plaplaceboundproofeq1}
f^{\prime}(t)=-2 \int \limits_{S}\gamma \left(|\nabla T_{\A}(t)u|_{n}^{p-2}\nabla T_{\A}(t)u-|\nabla T_{\A}(t)v|_{n}^{p-2}\nabla T_{\A}(t)v\right)\cdot(\nabla T_{\A}(t)u-\nabla T_{\A}(t)v)d\lambda,
\end{align}
for a.e. $t \in (0,\infty)$.
\end{lemma}
\begin{proof} Firstly, let us verify the desired local Lipschitz continuity. To this end, fix $c>0$ and note that $[0,c]\ni t \mapsto T_{\A}(t)u$ and $[0,c]\ni t \mapsto T_{\A}(t)v$ are by \cite[Lemma 7.8]{BenilanBook}, w.r.t. $||\cdot||_{L^{1}(S)}$, Lipschitz continuous. So let $C_{u},C_{v}\geq 0$ denote their Lipschitz constants. Then, we get for any $t_{1},t_{2} \in [0,c]$ that
\begin{align*}
|f(t_{1})-f(t_{2})|\leq |t_{1}-t_{2}|(C_{u}+C_{v})(2||u||_{L^{\infty}(S)}+2||v||_{L^{\infty}(S)}),
\end{align*}
where we used \cite[Lemma 3.3.3]{ich1}, which reads $||T_{\A}(t)w||_{L^{\infty}(S)}\leq ||w||_{L^{\infty}(S)}$ for all $t \geq 0$, $w \in L^{\infty}(S)$.\\
Consequently, $f$ is locally Lipschitz continuous and as it is real-valued, it is also differentiable almost everywhere.\\ 
Proof of (\ref{prop_plaplaceboundproofeq1}). Firstly, for all $w \in D(A)$ we have $T_{\A}(t)w \in D(A)$ and $-T^{\prime}_{\A}(t)w=A T_{\A}(t)w$ for a.e. $t \in (0,\infty)$, see \cite[Lemma 3.3.4]{ich1}. Thus, as $D(A)\subseteq W^{1,p}_{\gamma}(S)$, the integral occurring on the right-hand-side of (\ref{prop_plaplaceboundproofeq1}) exists for a.e. $t \in (0,\infty)$ and for almost every $t \in (0,\infty)$ we have $T_{\A}(t)u, T_{\A}(t)v \in D(A)$, $-T^{\prime}_{\A}(t)u=A T_{\A}(t)u$ and $-T^{\prime}_{\A}(t)v=A T_{\A}(t)v$. In light of this, it is intuitively clear that
\begin{align}
\label{prop_plaplaceboundproofeq3}
f^{\prime}(t)= -2 \int \limits_{S} (T_{\A}(t)u-T_{\A}(t)v)(AT_{\A}(t)u-AT_{\A}(t)v)d\lambda,
\end{align}
for a.e. $t \in (0,\infty)$ and making rigorous that one is allowed to perform the needed exchange of the integral and the differential works by the aid of dominated convergence and identical to the proof of \cite[Lemma 5.3]{ich1}.\\
Finally, (\ref{prop_plaplaceboundproofeq3}) implies (\ref{prop_plaplaceboundproofeq1}) by using $(T_{\A}(t)u-T_{\A}(t)v)$ as a test function in the definition of $A$.
\end{proof}

\begin{proposition}\label{prop_plaplacebound} Let $u,v \in L^{2}_{0}(S)$. Then we have
\begin{align}
\label{prop_plaplaceboundeq}
||T_{\A}(t)u-T_{\A}(t)v||_{L^{2}(S)}\leq \left(\kappa t+||u-v||_{L^{2}(S)}^{2-p}\right)^{\frac{1}{2-p}},~\forall t\geq 0,
\end{align}
where $\kappa:=(p-2)2^{2-p}\left(\int \limits_{S}\gamma^{\frac{2}{2-p}}d\lambda\right)^{\frac{2-p}{2}}C_{S,2}^{-p}$ and $C_{S,2}$ is the Poincar\'{e} constant of $S$ in $L^{2}_{0}(S)$.
\end{proposition}
\begin{proof} For now assume in addition $u,v \in L^{2}_{0}(S)\cap D(A)$, set $f(t):=||T_{\A}(t)u-T_{\A}(t)v||^{2}_{L^{2}(S)}$ for all $t \geq 0$ and let us derive an upper bound on $f^{\prime}$, which exists a.e. on $(0,\infty)$ due to Lemma \ref{lemma_lipdifae}.\\ 
Firstly, we have $W^{1,p}_{\gamma}(S)\subseteq W^{1,2}(S)$, since appealing to H\"older's inequality gives
\begin{align*}
\int \limits_{S} |\nabla \varphi|_{n}^{2}d\lambda = \int \limits_{S} \gamma^{-\frac{2}{p}}\gamma^{\frac{2}{p}} |\nabla \varphi|_{n}^{2}d\lambda \leq \left(\int \limits_{S} \gamma^{\frac{2}{2-p}}d\lambda\right)^{\frac{p-2}{p}}\left(\int \limits_{S} \gamma |\nabla \varphi|_{n}^{p}d\lambda\right)^{\frac{2}{p}}<\infty,~\forall \varphi \in W^{1,p}_{\gamma}(S).
\end{align*}
Consequently, employing Poincar\'{e}'s inequality yields
\begin{align}
\label{prop_plaplaceboundproofeq4}
\int \limits_{S}\gamma |\nabla \varphi|_{n}^{p}d\lambda \geq C_{S,2}^{-p}\left(\int \limits_{S} \varphi^{2}d\lambda\right)^{\frac{p}{2}}\left(\int \limits_{S} \gamma^{\frac{2}{2-p}}d\lambda\right)^{\frac{2-p}{2}},~\forall \varphi \in W^{1,p}_{\gamma}(S)\cap L^{2}_{0}(S).
\end{align}
Moreover, it is well known that $(|x|_{n}^{p-2}x-|y|_{n}^{p-2}y)\cdot(x-y)\geq 2^{2-p}|x-y|_{n}^{p}$ for all $x,y \in \mathbb{R}^{n}$, see \cite[Lemma 3.6]{plaplaceinequality}. By \cite[Lemma 3.3.4]{ich1}, we get $T_{\A}(t)u,T_{\A}(t)v \in W^{1,p}_{\gamma}(S)$ for a.e. $t \in (0,\infty)$ and by the aid of Lemma \ref{lemma_invsssg} we then obtain $T_{\A}(t)u-T_{\A}(t)v \in W^{1,p}_{\gamma}(S)\cap L^{2}_{0}(S)$ for a.e. $t \in (0,\infty)$. These observations enable us to conclude from (\ref{prop_plaplaceboundproofeq4})  and Lemma \ref{lemma_lipdifae} that
\begin{align*}
f^{\prime}(t)\leq - 2^{3-p} \int \limits_{S}\gamma |\nabla T_{\A}(t)u-\nabla T_{\A}(t)v|_{n}^{p}d\lambda \leq - 2^{3-p}C_{S,2}^{-p}\left(\int \limits_{S} \gamma^{\frac{2}{2-p}}d\lambda\right)^{\frac{2-p}{2}}f(t)^{\frac{p}{2}},
\end{align*}
for a.e. $t \in (0,\infty)$. Thus, by setting $\tilde{\rho}:= \frac{2}{p-2}$, we get $f^{\prime}(t)\leq -\kappa \tilde{\rho}f(t)^{1+\frac{1}{\tilde{\rho}}}$ for a.e. $t \in (0,\infty)$. Hence, invoking Lemma \ref{lemma_diffinequality} yields $f(t)\leq (\kappa t+f(0)^{-\frac{1}{\tilde{\rho}}})^{-\tilde{\rho}}$; thus by taking the square root and noting that $u,v\in L^{2}_{0}(S)\cap D(A)$ were arbitrary, we get
\begin{align}
\label{prop_plaplaceboundproofeq5}
||T_{\A}(t)u-T_{\A}(t)v||_{L^{2}(S)}\leq \left(\kappa t+||u-v||_{L^{2}(S)}^{2-p}\right)^{\frac{1}{2-p}},~\forall u,v \in L^{2}_{0}(S)\cap D(A),
\end{align}
for all $t \in [0,\infty)$. It remains to generalize the preceding inequality to all $u,v \in L^{2}_{0}(S)$. So fix $t \in [0,\infty)$, let $u,v \in L^{2}_{0}(S)$ and introduce $(u_{m})_{m \in \mathbb{N}},(v_{m})_{m \in \mathbb{N}}\subseteq D(A)$ such that $\lim \limits_{m \rightarrow \infty} u_{m}=u$ and $\lim \limits_{m \rightarrow \infty} v_{m}=v$ in $L^{2}(S)$, such sequences exist by \cite[Lemma 5.6]{ich1}. Now, one instantly verifies that $u_{m}-\overline{(u_{m})} \in D(A)$, with $A(u_{m}-\overline{(u_{m})})=A u_{m}$. Consequently, $u_{m}-\overline{(u_{m})} \in D(A)\cap L^{2}_{0}(S)$ for all $m \in \mathbb{N}$ and $\lim \limits_{m \rightarrow \infty} u_{m}-\overline{(u_{m})}=u$, in $L^{2}(S)$ since $\overline{(u)}=0$. Conclusively, as the analogous statements hold for $v,v_{m}$, (\ref{prop_plaplaceboundeq}) follows from (\ref{prop_plaplaceboundproofeq5}) and Lemma \ref{lemma_invsssg}.
\end{proof}

\begin{theorem}\label{theoremplaplace} Let $q \in [1,2]$ and let $(\eta_{k})_{k \in \mathbb{N}}\subseteq \mathcal{M}(\Omega;L_{0}^{q}(S))$ be an i.i.d. sequence. Moreover, let $(\beta_{m})_{m \in \mathbb{N}}$ be another i.i.d. sequence which is independent of $(\eta_{k})_{k \in \mathbb{N}}$ and assume that $\beta_{m}\sim Exp(\theta)$ for all $m \in \mathbb{N}$, where $\theta \in (0,\infty)$. In addition, assume $||\eta_{k}||_{L^{q}(S)} \in L^{2}(\Omega)$ for all $k \in \mathbb{N}$. Moreover, let $x \in \mathcal{M}(\Omega;L_{0}^{q}(S))$ be an independent initial, i.e. independent of  $((\eta_{k})_{k \in \mathbb{N}},(\beta_{k})_{k \in \mathbb{N}})$ and let \linebreak$\X_{x}:[0,\infty)\times \Omega \rightarrow L_{0}^{q}(S)$ be the process generated by $((\beta_{k})_{k \in \mathbb{N}},(\eta_{k})_{k \in \mathbb{N}},x,T_{\A})$ in $L_{0}^{q}(S)$.\\
Then $(\X_{x}(t))_{t \geq 0}$ is a time-homogeneous Markov process (w.r.t. the completion of its natural filtration) which possesses a unique invariant probability measure $\bar{\mu}:\B(L^{q}_{0}(S))\rightarrow [0,1]$. In addition, for any $\psi \in Lip(L_{0}^{q}(S))$, the convergence
\begin{align}
\label{theoremplaplace_eq1}
\lim \limits_{t \rightarrow \infty} \frac{1}{t}\int \limits_{0}\limits^{t}\psi(\X_{x}(\tau))d\tau=\int \limits_{L^{q}_{0}(S)} \psi(v)\bar{\mu}(dv):= \overline{(\psi)},
\end{align}
takes place with probability one, and if additionally $p \in (2,4)$, then there is a $\sigma^{2}(\psi)\in [0,\infty)$ such that
\begin{align}
\label{theoremplaplace_eq2}
\lim \limits_{t \rightarrow \infty }\frac{1}{\sqrt{t}}\left( \int \limits_{0}\limits^{t}\psi(\X_{x}(\tau))d\tau-t\overline{(\psi)}\right)=Y\sim N(0,\sigma^{2}(\psi)),
\end{align}
in distribution.
\end{theorem}
\begin{proof} By Lemma \ref{lemma_invsssg} $(T_{\A}(t))_{t \geq 0}$ is a time continuous, contractive semigroup on $L^{q}_{0}(S)$. Consequently, by choosing $V=L^{q}_{0}(S)$ in Section \ref{sec_mp} it follows from Theorem \ref{theorem_mp} that $\X_{x}$ is a time-continuous Markov process.\\
Moreover, it follows from Lemma \ref{lemma_invsssg} and Proposition \ref{prop_plaplacebound} that $(T_{\A}(t))_{t \geq 0}$ fulfills Assumption \ref{assumption}, where we choose $V=L^{q}_{0}(S)$, $W=L^{2}_{0}(S)$, $\rho:=\frac{1}{p-2}$ and $\kappa$ as in Proposition \ref{prop_plaplacebound}. Consequently, appealing to Proposition \ref{prop_uniqueinvpropmeas} yields the existence of a unique invariant probability measure and Theorem \ref{theorem_slln} implies (\ref{theoremplaplace_eq1}). Finally, (\ref{theoremplaplace_eq2}) follows from Theorem \ref{theorem_clt}, since $p \in (2,4)$ implies $\rho>\frac{1}{2}$.
\end{proof}

\begin{center}
	\textsc{Acknowledgment}
\end{center}
The present author is grateful to Prof. Dr. Alexei Kulik for fruitful conversations during a research stay of the present author at Technische Universit\"at Berlin.

%%%%%%%%%%%%%%%%%%%%%%%%%%%%%%%%
%%%%BEGIN bibliography%%%%%%%%%%%
%%%%%%%%%%%%%%%%%%%%%%%%%%%%%%%%

\end{document}